\documentclass[onefignum,onetabnum]{siamart171218}

\usepackage{lipsum}
\usepackage{amsfonts}
\usepackage{graphicx}
\usepackage{epstopdf}
\usepackage{algorithmic}
\ifpdf
  \DeclareGraphicsExtensions{.eps,.pdf,.png,.jpg}
\else
  \DeclareGraphicsExtensions{.eps}
\fi

\newcommand{\Z}{\mathbb{Z}}
\newcommand{\R}{\mathbb{R}}
\newcommand{\C}{\mathbb{C}}
\newcommand{\Co}{\mathcal{C}}
\newcommand{\A}{\mathcal{A}}
\newcommand{\X}{\mathcal{X}}
\newcommand{\Y}{\mathcal{Y}}
\newcommand{\E}{\mathcal{E}}
\newcommand{\ess}{\textnormal{ess}}

\newsiamremark{remark}{Remark}
\newsiamremark{hypothesis}{Hypothesis}
\crefname{hypothesis}{Hypothesis}{Hypotheses}
\newsiamthm{claim}{Claim}

\headers{Age-structured and delayed Lotka-Volterra models}{A. Perasso and Q. Richard}

\title{Asymptotic behavior of age-structured and delayed Lotka-Volterra models\thanks{Submitted to the editors DATE.}}
              
\author{Antoine Perasso\thanks{UMR 6249 Chrono-Environnement,
 Universit\'e Bourgogne Franche-Comt\'e, Besançon 25000, France. 
  (\email{antoine.perasso@univ-fcomte.fr}).}
\and Quentin Richard \thanks{UMR 5251 Institut de Mathématiques de Bordeaux, Universit\'e de Bordeaux, Talence 33400, France.
  (\email{quentin.richard@math.cnrs.fr}).}}

\usepackage{amsopn}

\ifpdf
\hypersetup{
  pdftitle={An Example Article},
  pdfauthor={D. Doe, P. T. Frank, and J. E. Smith}
}
\fi

\begin{document}

\maketitle

% REQUIRED
\begin{abstract}
In this work we investigate some asymptotic properties of an age-structured Lotka-Volterra model, where a specific choice of the functional parameters allows us to formulate it as a delayed problem, for which we prove the existence of a unique coexistence equilibrium and characterize the existence of a periodic solution. We also exhibit a Lyapunov functional that enables us to reduce the attractive set to either the nontrivial equilibrium or to a periodic solution. We then prove the asymptotic stability of the nontrivial equilibrium where, depending on the existence of the periodic trajectory, we make explicit the basin of attraction of the equilibrium. Finally, we prove that these results can be extended to the initial PDE problem.
\end{abstract}

% REQUIRED

% REQUIRED

\section{Introduction}\label{Sec:Intro}

Mathematical models describing the relationships between a predator and its prey are, since Lotka \cite{Lotka25} and Volterra \cite{Volterra26}, still a wide subject of study in population dynamics. Half a century later, Gurtin and Levine considered in \cite{Gurtin79} a model where the dynamics depend on the age of the interacting species. As introduced by Sharpe and Lotka in \cite{SharpeLotka11} and by McKendrick in \cite{McKendrick25}, structuring individuals according to a continuous age variable leads to the formulation of a linear PDE of transport type. Such models have been extensively studied by many researchers (see e.g. the books of Webb \cite{Webb85}, Iannelli \cite{Iannelli94}, Magal and Ruan \cite{Magal2008}, Inaba \cite{Inaba2017}).
Concerning the specific case of structured predator-prey models, one can see \cite{PerassoRichard19} for references. In this paper, we consider the following age-structured predator-prey system
\begin{equation}
\left\{
\begin{array}{rcl}
\partial_t x(t,a)+\partial_a x(t,a)&=&-\mu(a)x(t,a)- \gamma(a)y(t)x(t,a), \\
y'(t)&=& \alpha y(t) \int_0^{\infty}\gamma(a)x(t,a)\mathrm{d}a  -\delta y(t),\\
x(t,0)&=&\int_0^{\infty} \beta(a) x(t,a)da, \\
x(0,\cdot)&=&x_0(\cdot), \quad y(0)=y_0,
\end{array}
\right.
\label{Eq:PDEmodel}
\end{equation}
for every $t>0$ and $a>0$ with
$$(x_0,y_0)\in \Y_+:=L^1_+(\R_+)\times \R_+$$
where $x(t,a)$ and $y(t)$ respectively denote the density of prey at age $a$ and time $t$, and the density of predators at time $t$. Moreover, $\alpha\in (0,1)$ and $\delta>0$ are constant parameters that respectively denote the assimilation coefficient of ingested prey and the basic mortality rate of the predator. Finally $\mu, \gamma$ and $\beta \in L^\infty_+(\R_+)$ are nonnegative and age-dependent functions that represent the basic mortality rate, the predation rate and the birth rate of the prey. This model has already been analyzed in \cite{PerassoRichard19} by rewriting it as a Cauchy problem and using semigroup theory (see \cite{EngelNagel2000}, \cite{Webb85}). In \cite{PerassoRichard19}, we enlightened the existence of two thresholds:
$$R_0=\int_0^{\infty} \beta(a)e^{-\int_0^a \mu(s)\mathrm{d}s}\mathrm{d}a, \qquad R_-=\int_0^{a_1} \beta(a)e^{-\int_0^a \mu(s)\mathrm{d}s}\mathrm{d}a,$$
with
$$a_1=\sup\{a\geq 0 : |\text{supp}(\gamma)\cap (0,a)|=0\}<\infty$$
that enables the solutions to go extinct when $R_0<1$ and to explode when $R_->1$ (with initial conditions in some subspace of $\Y_+$). One can note that the age $a_1$ also corresponds to the minimum of the essential support of $\gamma$, that is the closed subset $S_\text{ess}(\gamma)\subset X$ given by
$$
S_\text{ess}(\gamma) := \bigcap \, \{F \text{ closed subset of } X, \, \gamma=0 \text{ a.e. on }  X\backslash F\}
$$ 
When
$$R_0>1 \quad \text{ and } \quad R_-<1,$$
numerical simulations suggest the possibility for the solutions to converge either to a periodic function, or to a nontrivial equilibrium denoted by $E_2$. The goal of the present paper is to prove the latter convergence in the particular case
\begin{equation}\label{Eq:PartCase}
\mu \equiv \mu_0, \quad \beta(a)=\beta_0 \mathbf{1}_{[\tau,\infty)}(a), \quad \gamma(a)=\gamma_0 \mathbf{1}_{[\tau,\infty)}(a)
\end{equation}
where $\mu_0, \beta_0, \gamma_0, \tau>0$ are some positive constants. In other words, we suppose the presence of a juvenile class that cannot be hunted. We can easily calculate
$$a_1=\tau, \qquad R_-=0, \qquad R_0=\dfrac{\beta_0 e^{-\mu_0 \tau}}{\mu_0},$$
and we suppose in the following that
\begin{equation}\label{Eq:R0>1}
R_0>1.
\end{equation}
Formal integrations of \eqref{Eq:PDEmodel} lead to
\begin{equation*}
\left\{
\begin{array}{lcl}
X'(t)&=&x(t,\tau)-\mu_0 X(t)-\gamma_0 y(t)X(t), \\
Z'(t)&=&x(t,0)-x(t,\tau)-\mu_0 Z(t), \\
y'(t)&=&\alpha \gamma_0 X(t)y(t)-\delta y(t),
\end{array}
\right.
\end{equation*}
for every $t\geq 0$, where
$$X(t)=\int_\tau^\infty x(t,a)da \quad \text{and} \quad Z(t)=\int_0^\tau x(t,a)da$$
are respectively the total quantity of prey older (resp. younger) than $\tau$ at time $t$. Using the boundary condition we get
$$x(t,0)=\beta_0 X(t)$$
for every $t\geq 0$ and
\begin{equation*}
x(t,\tau)=\begin{cases}
\beta_0 e^{-\mu_0 \tau}X(t-\tau) &\text{if } t\in[\tau,\infty), \\
\beta_0 e^{-\mu_0 t}x_0(\tau-t) &\text{if } t\in[0,\tau)
\end{cases}
\end{equation*}
so that we get the following delayed differential system
\begin{equation}\label{Eq:Delay}
\left\{
\begin{array}{rcl}
X'(t)&=& \beta_0 e^{-\mu_0 \tau} X(t-\tau)-\mu_0 X(t)-\gamma_0 X(t)y(t), \\
Z'(t)&=&\beta_0 X(t)-\beta_0 e^{-\mu_0 \tau}X(t-\tau)-\mu_0 Z(t), \\
y'(t)&=& \alpha \gamma_0 X(t)y(t) -\delta y(t)
\end{array}
\right.
\end{equation}
for any $t\geq \tau$. We note that the equivalence between \eqref{Eq:PDEmodel} and \eqref{Eq:Delay} is true only if the delayed differential system \eqref{Eq:Delay} is equipped with the initial condition:
$$X(\theta)=\phi(\theta), \qquad Z(\theta)=\psi(\theta), \qquad y(0)=y_0$$
for every $\theta\in[0,\tau]$, where $\phi$ and $\psi$ are solutions of the following ODE for any $t\in[0,\tau]$:
\begin{equation*}
\left\{
\begin{array}{rcl}
\phi'(t)&=&\beta_0 e^{-\mu_0 t} x_0(\tau-t)-\mu_0 \phi(t)-\gamma_0 \phi(t)y(t), \\
\psi'(t)&=&\beta_0 \phi(t) -\beta_0 e^{-\mu_0 t}x_0(\tau-t)-\mu_0 \psi(t), \\
y'(t)&=&\alpha \gamma_0 \phi(t)y(t)-\delta y(t), \\
\phi(0)&=&\int_\tau^\infty x_0(a)da, \quad \psi(0)=\int_0^\tau x_0(a)da, \quad y(0)=y_0.
\end{array}
\right.
\end{equation*}
Since we can solve $X$ and $y$ independently of $Z$ in \eqref{Eq:Delay}, we will focus in the following on the delayed Lotka-Volterra system:
\begin{equation}\label{Eq:System}
\left\{
\begin{array}{rcl}
X'(t)&=& \beta_0 e^{-\mu_0 \tau} X(t-\tau)-\mu_0 X(t)-\gamma_0 X(t)y(t), \\
y'(t)&=& \alpha \gamma_0 X(t)y(t) -\delta y(t).
\end{array}
\right.
\end{equation}
for any $t\geq \tau$, and we will consider the more general case by taking an arbitrary initial condition $(\phi,y_\tau)\in\Co([0,\tau],\R_+)\times \R_+$, \textit{i.e.} such that
$$X(\theta)=\phi(\theta), \ \forall \theta\in[0,\tau], \qquad y(\tau)=y_\tau.$$
Note that the more general case where
$$\mu\in L^\infty(0,\infty), \qquad {\mu}_{\mid [\tau, \infty)}\equiv \mu_0, \quad \mu_0>0$$
could be easily extended to obtain a similar delay differential system as \eqref{Eq:System}. In the latter model, the delay can be seen as some latency for the prey to reproduce. Concerning Lotka-Volterra equations, delay was first introduced by May \cite{May73b} in a vegetation-herbivore-carnivore context, to model the time for the vegetation to recover. Thereafter, many authors studied similar delayed models (see some references in the general books of Cushing \cite{Cushing77}, Kuang \cite{Kuang93}, Arino et al \cite{Arino2007} and Smith \cite{Smith2010}). Some of the papers concern the global stability of equilibria (see e.g. \cite{Beretta88}, \cite{BerettaKuang96}, \cite{He98}, \cite{Saito99}, \cite{Saito2001}, \cite{Muroya2003}, \cite{ShiChen2017}).

However, in the papers mentioned above, a carrying capacity is present in the prey equation, meaning that prey grow logistically instead of exponentially. A consequence of this assumption is that, in absence of delay, the nontrivial equilibrium is asymptotically stable for some range of parameters. Adding some delay can then destabilize the equilibrium and make periodic solutions appear from a Hopf bifurcation (see e.g. \cite{YanLi2006}, \cite{Peng2017} and also \cite{Faria2001}, \cite{Yan2017} when adding some diffusion).

In our case, when the delay is equal to zero, \eqref{Eq:System} becomes the classical ODE Lotka Volterra model, so the coexistence equilibrium is only stable but not asymptotically stable. We show that, contrarily to the other papers, adding some delay in the reproduction term of the prey do not destabilize the coexistence equilibrium but make it become asymptotically stable, under technical assumptions.

The method used to prove this convergence is based on the existence of a Lyapunov function (see \cite{Gurel68} or more recently \cite{Hsu2005} for surveys of such functions in various ecological ODE and reaction-diffusion models). When dealing with global stability of positive equilibria, many suitable Lyapunov functionals are defined using the following key function:
\begin{equation}\label{Eq:G}
\begin{array}{rcl}
g:(0,\infty)&\longrightarrow &\R \\
x&\longmapsto& x-\ln(x)-1.
\end{array}
\end{equation}
The latter has been first used by Goh \cite{Goh77} in a context of a multi-species ODE Lotka-Volterra model. Hsu established similar Lyapunov functions in \cite{Hsu78} for models with more general functional responses. One may also see \cite{Goh79} for a model of mutualism.

For the present model, one shall also use the following Volterra-type Lyapunov functional that incorporates the delay term:
\begin{equation*}
\begin{array}{rcl}
V:\Co([0,\tau])&\longrightarrow &\R \\
\phi &\longmapsto& \displaystyle \int_{0}^\tau g\left(\dfrac{\phi(s)}{X^*}\right)ds
\end{array}
\end{equation*}
where $(X^*,y^*)$ is the nontrivial equilibrium. The latter was first introduced the same year in \cite{HuangTakeuchi2010}, \cite{LiShu2010}, 
\cite{McCluskey2010}, \cite{McCluskey2010b} for epidemiological models (see also \cite{MagCluskWebb2010}, \cite{Perasso2019} and the references therein for similar functional in structured populations PDE models). Concerning Lotka-Volterra models, a few papers used this functional: \cite{Vargas2012}, \cite{Vargas2015b}, \cite{Vargas2015}  and \cite{HuangLiu2016}.

In contrast to the papers mentioned previously, in our case the attractive sets are not reduced to the equilibrium, but are given by a set of periodic solutions, where the period is exactly equal to the delay. Consequently, one can \textit{a priori} only state the convergence to either the equilibrium or to an eventual $\tau$-periodic solution. Using properties on the period of the solutions of the classical Lotka-Volterra ODE model, we show that a necessary and sufficient condition to get such periodic solution is the following:
\begin{equation}\label{Eq:Hyp_Period}
\dfrac{\tau \sqrt{\delta y^* \gamma_0}}{2 \pi}>1.
\end{equation}
When \eqref{Eq:Hyp_Period} is not satisfied, then the global asymptotic stability of the nontrivial equilibrium is proved for `positive' initial conditions. In the case where \eqref{Eq:Hyp_Period} holds, we exhibit an attractive set in which the equilibrium is globally asymptotically stable.

The paper is structured as follows: in Section \ref{Sec:Prel}, we state the framework used in the following and we decompose the space of initial conditions into invariant spaces. In Section \ref{Sec:Attrac}, we exhibit a Lyapunov function and we prove an asymptotic stability result for the nontrivial equilibrium when \eqref{Eq:Hyp_Period} does not hold. In the case where \eqref{Eq:Hyp_Period} holds, even if the existence of a periodic solution is ensured, we prove that this latter is unattractive and the asymptotic stability of the nontrivial equilibrium in a suitable basin of attraction defined from the Lyapunov function. The two cases are enlightened by numerical simulations. Finally, in Section \ref{Sec:PDE}, we deduce asymptotic results for the initial PDE problem \eqref{Eq:PDEmodel}.

\section{Preliminaries}\label{Sec:Prel}

\subsection{Framework and definitions}\label{Sec:Framework}

Let the Banach space
$$\X=\Co([0,\tau],\R)\times \R$$
endowed with the norm
$$\|(u,v)\|_{\X}=\|u\|_\infty+|v|$$
and let $\X_+$ be its nonnegative cone. We study \eqref{Eq:System} with the initial condition
\begin{equation*}
X(\theta)=\phi(\theta), \ \forall \theta \in [0,\tau], \quad y(\tau)=y_\tau,
\end{equation*}
where $(\phi,y_\tau)\in \X$. The equilibria of \eqref{Eq:System} are given by
$$E_0:=(0,0) \in \R^2; \qquad E^*:=(X^*,y^*)=\left(\dfrac{\delta}{\alpha \gamma_0}, \dfrac{\beta_0 e^{-\mu_0 \tau}-\mu_0}{\gamma_0}\right)\in\R^2.$$
We verify that $E^*$ exists (in the positive orthant) if and only if \eqref{Eq:R0>1} holds, and the nontrivial equilibrium is unique under this latter condition.

The initial-value problem \eqref{Eq:System} can be written as the following abstract Cauchy problem:
\begin{equation}
\label{Eq:Sys}
\left\{
\begin{array}{rcl}
\begin{pmatrix}
X \\
y
\end{pmatrix}'(t)&=&f(X_t,y(t)), \ \forall t \geq \tau \\
X_\tau&=&\phi, \qquad y(\tau)=y_\tau
\end{array}
\right.
\end{equation}
where $(\phi,y_\tau)\in \X$ and $f:\X\to \R^2$ is defined by
\begin{equation*}
f(\phi,y)=\begin{pmatrix}
\beta_0 e^{-\mu_0 \tau}\phi(0)-\mu_0 \phi(\tau)-\gamma_0 \phi(\tau)y\\
\alpha \gamma_0 \phi(\tau)y -\delta y
\end{pmatrix},
\end{equation*}
and where
$$X_t(\theta):=X(t+\theta-\tau), \ \forall\theta \in [0,\tau]$$
(so that $X_t(0)=X(t-\tau)$, $X_t(\tau)=X(t)$ and $X_\tau(\theta)=X(\theta)$ for any $\theta\in[0,\tau]$). We omit the initial condition dependence since there is no misunderstanding, so we write $X_t(\theta)$ instead of $X_t(\theta,z)$, where $z:=(\phi,y_0)$.
We start by giving an existence and uniqueness result.
\begin{proposition}\label{Prop:Semiflow}
For every initial condition $z:=(\phi,y_\tau)\in \X_+$, Problem \eqref{Eq:Sys} has a unique mild solution $(X_t,y(t))$ for every $t\geq \tau$. Moreover, Problem \eqref{Eq:Sys} induces a continuous semiflow via:
\begin{equation*}
\begin{array}{rcl}
\Phi:[\tau,\infty)\times \X_+ &\to& \X_+ \\
(t,z)&\mapsto& \Phi_t(z):=(X_t,y(t)).
\end{array}
\end{equation*}
\end{proposition}

\begin{proof}
The proposition results from the general case \cite[Proposition 3.2]{PerassoRichard19}.
\end{proof}

\begin{remark}\label{Remark:Explosion}
Consequently of the latter proposition, the solution remains in the nonnegative cone and there is no explosion in finite time.
\end{remark}
In what follows, we shall use the notations:
$$\overline{E_0}:=(0,0)\in \X, \qquad \overline{E^*}:=(X^*\mathbf{1}_{[0,\tau]},y^*)\in \X.$$

One of the goal of this article is to investigate some stability and attractiveness properties of $\overline{E^*}$. We therefore remind the following definitions:
\begin{definition}
Let $S\subset \X$ be a subset of $\X$. We say that $E^*$ is
\begin{itemize}
\item \textbf{(Lyapunov) stable} if for every $\varepsilon>0,$ there exists $\eta>0$ such that
$$\|z-\overline{E^*}\|_\X \leq \eta \quad \Rightarrow \quad \|\Phi_t(z)-\overline{E^*}\|_\X\leq \varepsilon, \quad \forall t\geq 0;$$
\item \textbf{locally attractive in $S$} if there exists $\eta>0$ such that for every $z\in S$ satisfying $\|z-\overline{E^*}\|_{\X}\leq \eta$, then
\begin{equation}\label{Eq:LocAttrac}\lim_{t\to \infty}\|\Phi_t(z)-\overline{E^*}\|_{\X}=0,
\end{equation}
\textit{i.e.}
$$\lim_{t\to \infty}y(t)=y^*, \qquad \lim_{t\to \infty}X(t)=X^*;$$
\item \textbf{locally asymptotically stable in $S$} if $E^*$ is stable and locally attractive in $S$;
\item \textbf{globally attractive in $S$} if for every $z\in S$, \eqref{Eq:LocAttrac} is satisfied;
\item \textbf{globally asymptotically stable in $S$} if $E^*$ is stable and globally attractive in $S$.
\end{itemize}
\end{definition}

\subsection{Partition of $\X_+$}\label{Sec:Partition}

Consider the sets
$$ S_0=\{(\phi,y)\in \X_+ : \int_{0}^\tau\phi(a)\mathrm{d}a>0\}, \quad \partial S_0=\X_+\setminus S_0, $$
$$ S_1=\{(\phi,y)\in \X_+ : \phi(a)>0, \ \forall a\in[0,\tau] \}, $$
$$ S_2=\{(\phi,y)\in \X_+ : y>0, \int_{0}^\tau \phi(a)\mathrm{d}a>0\}, \quad \partial S_2=\X_+ \setminus S_2,$$
$$ S_3=\{(\phi,y)\in \X_+ : y>0, \phi(a)>0, \ \forall a\in[0,\tau]\}.$$
In order to make an analogy with the PDE model \eqref{Eq:PDEmodel}, note that for any $a\in[0,\tau]$:
$$\phi(a)=X(a)=\int_\tau^\infty x(a,s)ds=e^{-\mu_0 a}\int_\tau^\infty x_0(s-a)e^{-\gamma_0 \int_\tau^s y(a+\tau-u)du}ds$$
so that
$$c e^{-\mu_0 a}\int_{\tau-a}^\infty x_0(s)ds\leq \phi(a)\leq e^{-\mu_0 a}\int_{\tau-a}^\infty x_0(s)ds$$
for some constant $c>0$. Consequently $\phi(a)>0$ if and only if $\int_{\tau-a}^\infty x_0(s)ds>0$ so that
\begin{equation*}
\left\{
\begin{array}{lcl}
&\phi(a)>0, \ \forall a\in[0,\tau] &\Longleftrightarrow \int_\tau^\infty x_0(s)ds>0 \vspace{0.1cm} \\
&\int_0^\tau \phi(a)da>0 &\Longleftrightarrow \int_0^\infty x_0(s)ds>0,
\end{array}
\right.
\end{equation*}
Hence, $S_0$ is the set where there is initially a nonzero total number of preys, while $S_1$ consists of initial conditions with preys older than $\tau$. Similarly $S_2$ and $S_3$ have respectively the same meaning as $S_0$ and $S_1$, but with initially a positive quantity of predators.

\begin{remark}\label{Rem:Partition}
We have the inclusions
$$ S_3\subset S_2 \subset S_0, \quad S_3\subset S_1 \subset S_0$$
and we get the partition
$$\X_+ = S_2 \sqcup (\partial S_2 \cap S_0) \sqcup (\partial S_2 \cap \partial S_0)$$
(disjoint unions) that is actually
$$\X_+ = S_2 \sqcup (\partial S_2 \cap S_0) \sqcup \partial S_0$$
since $\partial S_0 \subset \partial S_2$.
\end{remark}

\subsection{Invariant sets}

We start by reminding some definitions.
\begin{definition}\label{Def:Orbit}
Denote by $\mathcal{O}_z=\{\Phi_t(z),t\geq \tau\}$ the orbit starting from $z\in \X_+$ and
$$ \omega(z)=\underset{s \geq \tau}{\cap}\overline{\{\Phi_t(z), t\geq s\}}$$
the $\omega-\text{limit}$ set of $z$.
\end{definition}

\begin{definition}
Let $S, T\subset \X$ and $s\geq 0$, then in all the following we will say that $S$ is
\begin{enumerate}
\item \textbf{positively invariant} if $\Phi_t(S)\subset S$ for every $t\geq \tau$, \textit{i.e.} for every $z\in S$ and every $t\geq \tau$, $\Phi_t(z)\in S$;
\item $(s,T)$-\textbf{positively invariant} if for every $z\in S$, then $\Phi_t(z)\in T$ for every $t\geq s+\tau$.
\end{enumerate}
\end{definition}

\begin{remark}
In all the following, we will denote by $(X_t,y(t))\in \X$ the solution of \eqref{Eq:Sys} at time $t\geq \tau$ with initial condition $(\phi,y_\tau)\in \X$.
\end{remark}
We now give some properties about the sets defined in Section \ref{Sec:Partition}, with first a useful lemma.

\begin{lemma}\label{Lemma:Positivity}
Let $(\phi,y_\tau)\in \X_+$ be a nonnegative initial condition. If there exists $t^*\in [0,\tau]$ such that $\phi(t^*)>0$ then $X(t^*+\tau)>0$.
\end{lemma}

\begin{proof}
By contradiction, suppose that $X(t^*+\tau)=0$ then Equation \eqref{Eq:System} implies
$$X'(t^*+\tau)=\beta_0 e^{-\mu_0 \tau}X(t^*)>0,$$
which contradicts the nonnegativity of $X$.
\end{proof}

\begin{proposition}\mbox{}\label{Prop:Invariance}
\begin{enumerate}
\item The sets $S_1$ and $S_3$ are positively invariant.
\item The set $S_0$ (resp $S_2$) is $(2\tau,S_1)$-positively invariant (resp $(2\tau,S_3)$). 
\item The set $\partial S_0$ is positively invariant and the equilibrium $\overline{E_0}$ is globally attractive in $\partial S_0$ .\item The set $\partial S_2$ is positively invariant. Moreover, if we take the restriction of $\Phi$ to the set $S_0 \cap \partial S_2$, then the solution $(X,y)$ of Problem \eqref{Eq:System} goes to $(\infty,0)$ when $t\to \infty$.
\end{enumerate}
\end{proposition}

\begin{proof}\mbox{}
\begin{enumerate}
\item Consider an initial condition $(\phi,y_\tau)\in S_1$. Then $X_\tau=\phi$ and
$$X(t)=\phi(t)>0, \ \forall t\in [0,\tau].$$
Lemma \ref{Lemma:Positivity} implies that $X(t)>0$ for every $t\in[\tau,2\tau]$. Repeating this argument, we get
$$X(t)>0, \ \forall t\geq 0.$$ Consequently $S_1$ is positively invariant. 
We then easily see that $S_3$ is positively invariant since $y'(t)\geq -\delta y(t)$ for $t\geq \tau$ and
$$y(t)\geq y_\tau e^{-\delta (t-\tau)}>0,\quad \forall t\geq \tau$$
when $(\phi,y_\tau)\in S_3$.

\item Take an initial condition $(\phi, y_\tau)\in S_0$. We then have $\int_0^\tau \phi(a)\mathrm{d}a>0$ so there exists $t^* \in [0,\tau]$ such that
$$\phi(t^*)>0.$$
Using Lemma \ref{Lemma:Positivity}, we get $$X(t^*+\tau)>0.$$
Since we have
$$X'(t)\geq -(\mu_0+\gamma_0 y(t))X(t), \ \forall t\in[t^*+\tau,3\tau],$$
it follows that
$$X(t)\geq X(t^*+\tau)e^{-(\mu_0+\gamma_0 \overline{y})[t-(t^*+\tau)]}>0, \ \forall t\in [t^*+\tau,3\tau],$$
where
$$\overline{y}=\max_{t\in [t^*+\tau,3\tau]}y(t)<\infty$$
using Remark \ref{Remark:Explosion}. We then have
$$X(t)>0, \ \forall t\in [2\tau,3\tau],$$
and $(X_{3\tau},y(3\tau))\in S_1$. With the first point, we can see that $S_0$ (resp. $S_2$) are $(2\tau,S_1)$ (resp. $(2\tau, S_3)$)-positively invariant.

\item Consider an initial condition $(\phi,y_\tau)\in \partial S_0$. We have $\int_0^\tau \phi(a)\mathrm{d}a=0$ and
$$X(t)=0, \ \forall t\in[0,\tau],$$
which leads to
$$X'(t)=-\mu_0 X(t) -\gamma_0 X(t)y(t)\leq 0,\quad \forall t\in [\tau,2\tau],$$
so $X$ is nonincreasing on $[\tau,2\tau]$. Since $X$ is nonnegative, then
$$X(t)=0, \ \forall t \in [\tau,2\tau].$$ Repeating this argument, we get $X(t)=0$ for every $t\geq \tau$. We readily see that 
$$\int_0^\tau X_t(\theta)\mathrm{d}\theta=0$$
for every $t\geq \tau$ and $\partial S_0$ is positively invariant. Since 
$$X(t)=0, \ \forall t\geq \tau$$
and $y'(t)\leq -\delta y(t)$ for every $t\geq \tau$, it is then clear that 
$$\lim_{t\to \infty} y(t)=0,$$
whence the solution of \eqref{Eq:System} converge to $E_0$.

\item We know that $\partial S_2 \cap \partial S_0=\partial S_0$ is positively invariant. Considering an initial condition $(\phi, y_\tau)\in \partial S_2\cap S_0$ we get
$$y_\tau=0 \ \text{ and } \ \int_0^\tau \phi(a)\mathrm{d}a>0.$$
Then \eqref{Eq:System} implies that 
$$y(t)=0, \ \forall t\geq \tau.$$
Since $S_0$ is positively invariant, we get the invariance of $\partial S_2\cap S_0$ and $\partial S_2$. Moreover, with the second and third points, we have
$$\Phi_t(\phi,y_\tau)\in S_1\cap \partial S_2, \ \forall t\geq 3\tau,$$
whence
$$X(t)>0, \ \forall t\geq 2\tau.$$
We see that \eqref{Eq:System} becomes the delayed Malthusian equation
$$X'(t)=\beta_0 e^{-\mu_0 \tau}X(t-\tau)-\mu_0 X(t), \ \forall t\geq \tau.$$
Such class of equation has been studied in \cite[Sections 2.1 and 2.2]{Iannelli2014}, where the authors proved that the solution behaves as
$$X(t)=c_0 e^{\alpha^* t}(1+\Omega(t)), \qquad \lim_{t\to \infty}\Omega(t)=0,$$
where $c_0>0$ and $\alpha^*>0$ whenever \eqref{Eq:R0>1} holds.
Consequently we get
$$\lim_{t\to \infty}X(t)=\infty.$$
\end{enumerate}
\end{proof}

\begin{remark}
Consequently to Proposition \ref{Prop:Invariance}, 2), all the asymptotic results proved for initial conditions in $S_3$ can be extended to $S_2$.
\end{remark}

Note that the behavior of the solutions when considering an initial condition in $\partial S_2\cap S_0$ or $\partial S_0$ is clear. By means of Remark \ref{Rem:Partition} and the latter proposition, it remains to prove what happens when the initial condition is taken in $S_3$.

\section{Asymptotic behavior}\label{Sec:Attrac}

In this section, we deal with the asymptotic behavior of the solutions.

\subsection{Local asymptotic stability of $E^*$}

We start by handling the local stability of the nontrivial equilibrium. Linearising \eqref{Eq:System} around $E^*$ gives:
\begin{equation}\label{Eq:Linear}
\left\{
\begin{array}{rcl}
X'(t)&=&\beta_0 X(t-\tau)e^{-\mu_0 \tau}-\mu_0 X(t)-\gamma_0 X^*y(t)-\gamma_0 X(t)y^*, \\
y'(t)&=&\alpha \gamma_0 y^* X(t)
\end{array}
\right.
\end{equation}
that can be rewritten under the form
$$\begin{pmatrix}
X'(t) \\
y'(t)
\end{pmatrix}=A_1 \begin{pmatrix}
X(t) \\
y(t)
\end{pmatrix}+A_2 \begin{pmatrix}
X(t-\tau) \\
y(t-\tau)
\end{pmatrix}$$
with
$$A_1=\begin{pmatrix}
-\mu_0-\gamma_0 y^* & - \gamma_0 X^* \\
\alpha \gamma_0 y^* & 0
\end{pmatrix}, \qquad A_2=\begin{pmatrix}
\beta_0 e^{-\mu_0 \tau} & 0 \\
0 & 0
\end{pmatrix}.$$
The characteristic equation of \eqref{Eq:Linear} is classically given by (see \textit{e.g.} \cite{ACR2006} or \cite{Smith2010}):
$$\det\left(\lambda I-A_1 -A_2 e^{-\lambda \tau}\right)=0$$
which reduces (by definition of $(X^*,y^*)$ to
\begin{equation}\label{Eq:LocalStab}
\kappa_1(\lambda)+\kappa_2(\lambda)e^{-\lambda \tau}=0.\\
\end{equation}
where $\kappa_1$ and $\kappa_2$ are given by
\begin{equation}\label{Eq:kappa}
\left\{
\begin{array}{rcl}
\kappa_1(\lambda)&=&\lambda^2+\lambda \beta_0 e^{-\mu_0 \tau}+\delta \gamma_0 y^*, \\
\kappa_2(\lambda)&=&-\lambda \beta_0 e^{-\mu_0 \tau}.
\end{array}
\right.
\end{equation}

\begin{theorem}\label{Thm:LocStab}
Every solution of \eqref{Eq:LocalStab} has non positive real part. Moreover, \eqref{Eq:LocalStab} has two purely imaginary roots if and only if 
\begin{equation}\label{Eq:Cond_Imag}
\dfrac{\tau \sqrt{\delta y^* \gamma_0}}{2\pi}\in \Z
\end{equation}
holds. In this case, the roots are given by
\begin{equation}\label{Eq:Lambda}\lambda_\pm=\pm i\sqrt{\delta y^* \gamma_0}.
\end{equation}
Consequently, if \eqref{Eq:Cond_Imag} does not hold, then $E^*$ is locally asymptotically stable for \eqref{Eq:System}.
\end{theorem}

Before proving the theorem, let us remind let a result (see \cite{Smith2010} Proposition 4.9) about absolute stability.

\begin{proposition}\label{Prop:StabAbs}
Let $\kappa_1, \kappa_2$ be two polynomial functions with real coefficients satisfying the equation \eqref{Eq:LocalStab} and suppose that:
\begin{enumerate}
\item $\kappa_1(\lambda)\neq 0, \Re(\lambda)\geq 0$.
\item $|\kappa_2(iy)|<|\kappa_1(iy)|, 0\leq y<\infty$.
\item $\lim_{|\lambda|\to \infty, \Re(\lambda)\geq 0} |\kappa_2(\lambda)/\kappa_1(\lambda)|=0$.
\end{enumerate}
Then every root $\lambda$ of \eqref{Eq:LocalStab} satisfies $\Re(\lambda)<0$ for every $\tau \geq 0$.
\end{proposition}

\begin{proof}(Theorem \ref{Thm:LocStab}.)
Let us check the hypotheses of Proposition \ref{Prop:StabAbs}. \\
\begin{enumerate}
\item Let $\lambda=\lambda_r +i \lambda_i$. Then 
$$\kappa_1(\lambda)=\lambda_r^2-\lambda_i^2+\beta_0 \lambda_r e^{-\mu_0 \tau}+\delta y^*\gamma_0+2i \lambda_i \lambda_r +i \beta_0\lambda_i e^{-\mu_0 \tau}.$$
Thus we have
\begin{equation*}
\kappa_1(\lambda)=0 \Longleftrightarrow\left\{
\begin{array}{lcl}
\lambda_r^2-\lambda_i^2+\beta_0 \lambda_r e^{-\mu_0 \tau}+\delta y^* \gamma_0=0, \\
2\lambda_r\lambda_i+\beta_0 e^{-\mu_0 \tau}\lambda_i=0.
\end{array}
\right.
\end{equation*}
The second equation gives us 
$$\lambda_i=0 \quad \text{ or } \quad 2\lambda_r+\beta_0 e^{-\mu_0 \tau}=0.$$
If $\lambda_i=0$, then we have
$$\lambda_r^2+\beta_0 \lambda_r e^{-\mu_0 \tau}+\delta y^* \gamma_0=0$$
and the latter equation has no nonnegative solution. If
$$2\lambda_r +\beta_0e^{-\mu_0 \tau}=0,$$
then necessarily $\lambda_r<0$. Consequently the first condition is satisfied.
\item We now compute the limit. We have
\begin{equation*}
\begin{array}{lcl}
&&\dfrac{|\kappa_2(\lambda)|^2}{|\kappa_1(\lambda)|^2}\\
&=&\dfrac{(\beta_0 e^{-\mu_0\tau})^2(\lambda_r^2+\lambda_i^2)}{(\lambda_r^2-\lambda_i^2+\beta_0 \lambda_r e^{-\mu_0 \tau}+\delta y^* \gamma_0)^2+(2\lambda_i\lambda_r+\beta_0\lambda_i e^{-\mu_0\tau})^2}.
\end{array}
\end{equation*}
The denominator is thus equal to 
$$(\lambda_r^2+\lambda_i^2)^2+M-2\lambda_i^2 \delta y^* \gamma_0,$$ where
\begin{equation*}
\begin{array}{lcl}
M&:=&(\beta_0 \lambda_r e^{-\mu_0 \tau})^2+(\delta y^* \gamma_0)^2+(\beta_0 \lambda_i e^{-\mu_0 \tau})^2+2\lambda_i^2 \lambda_r \beta_0 e^{-\mu_0 \tau} \\
&&+ 2\lambda_r^3\beta_0 e^{-\mu_0 \tau}+2\lambda_r^2 y^*\delta \gamma_0+2\beta_0 \lambda_r e^{-\mu_0 \tau}\delta y^* \gamma_0
\end{array}
\end{equation*}
so $M\geq 0$ when $\Re(\lambda)\geq 0$.
Consequently we have:
\begin{equation*}
\begin{array}{rcl}
\dfrac{|\kappa_2(\lambda)|^2}{|\kappa_1(\lambda)|^2}\leq\dfrac{(\beta_0 e^{-\mu_0 \tau})^2}{(\lambda_r^2+\lambda_i^2)-2\dfrac{\lambda_i^2 \delta y^* \gamma_0}{\lambda_r^2+\lambda_i^2}} \leq \dfrac{(\beta_0 e^{-\mu_0 \tau})^2}{\lambda_r^2+\lambda_i^2-2\delta y^* \gamma_0}\xrightarrow[\Re(\lambda)\geq 0]{|\lambda|\to \infty}0.
\end{array}
\end{equation*}
Thus the third condition is satisfied. 
\item We know that
$$|\kappa_2(iy)|=y\beta_0 e^{-\mu_0 \tau}$$
and
$$|\kappa_1(iy)|=\sqrt{(\delta y^* \gamma_0-y^2)^2+(y \beta_0 e^{-\mu_0 \tau})^2}.$$
Thus for every $y\geq 0$, we have
$$|\kappa_2(iy)|\leq |\kappa_1(iy)|$$
and there is equality only when 
$$\delta y^*\gamma_0=y^2,$$
which means
$$y=\sqrt{\delta y^* \gamma_0}.$$
Consequently the second condition is \textit{not} totally satisfied but by slightly modifying the system, we can avoid the problem. Following the sketch of proof of Section 3 in \cite{Brauer87}, we consider the following characteristic equation, for $\varepsilon>0$ small enough:
\begin{equation}\label{Eq:CaracModif}
\kappa_1(\lambda)+\varepsilon \kappa_1(\lambda)+\kappa_2(\lambda)e^{-\lambda \tau}=0.
\end{equation}
Thus the hypotheses of Proposition \ref{Prop:StabAbs} are satisfied for \eqref{Eq:CaracModif}, hence all roots of \eqref{Eq:CaracModif} have negative real part for all $\varepsilon>0$ small enough. Since the roots of \eqref{Eq:CaracModif} continuously depend of $\varepsilon$, then all roots of $\eqref{Eq:LocalStab}$ have non positive real part. Let $\lambda=i \omega, \omega>0$. Then $\lambda$ verifies the equation \eqref{Eq:LocalStab} if and only if
$$-\omega^2+i\omega\beta_0 e^{-\mu_0 \tau} +\delta y^* \gamma_0=i\omega \beta_0 e^{-\mu_0 \tau}e^{-i\omega\tau}.$$
Considering the real and imaginary parts, we get the following system:
\begin{equation*}
\left\{
\begin{array}{rcl}
-\omega^2+\delta y^* \gamma_0&=&\omega\beta_0 e^{-\mu_0 \tau}\sin(\omega \tau), \\
\omega \beta_0  e^{-\mu_0 \tau}&=&\omega \beta_0 e^{-\mu_0 \tau}\cos(\omega \tau).
\end{array}
\right.
\end{equation*}
\begin{equation*}
\Leftrightarrow
\left\{
\begin{array}{rcl}
\omega^2&=&\delta y^* \gamma_0, \\
\cos(\omega \tau)&=&1.
\end{array}
\right.
\end{equation*}
Consequently, there are purely imaginary roots of \eqref{Eq:LocalStab} if and only if \eqref{Eq:Cond_Imag} is satisfied. 
\end{enumerate}
\end{proof}

\subsection{Lyapunov function}

Now we want to get the global attractiveness of $E^*$ on some subset $S\subset \X$. To this end, we use Lyapunov functionals. Let
$$L_*(\phi,y)=V_1(\phi,y)+V_2(\phi,y)+V_3(\phi,y)$$
formally defined for $(\phi,y)\in \X$ by
$$V_1(\phi,y)=\alpha X^*g\left(\dfrac{\phi(\tau)}{X^*}\right),$$
$$V_2(\phi,y)=y^*g\left(\dfrac{y}{y^*}\right),$$
$$V_3(\phi,y)=\alpha \beta_0 e^{-\mu_0 \tau}X^*\int_0^\tau g\left(\dfrac{\phi(s)}{X^*}\right)\mathrm{d}s,$$
where $g$ is defined by \eqref{Eq:G}. One may observe that the function $V_1+V_2$ is the one used in the classical Lotka-Volterra ODE model to prove the periodicity of the solutions. Note that the fact that
\begin{equation}\label{Eq:Expl-Lyap}
\lim_{\phi(\tau)\to 0}L_*(\phi,y)=\infty, \qquad \lim_{y\to 0}L_*(\phi,y)=\infty
\end{equation}
will play an important role in the next section.

\begin{proposition}\label{Prop:WellDefined}
The function $(t,w)\mapsto L_*(\Phi_t(w))$ is well-defined on $[\tau,\infty)\times S_3$ whenever \eqref{Eq:R0>1} holds.
\end{proposition}

\begin{proof}
Note that the assumption \eqref{Eq:R0>1} is necessary to define $L_*$ since the equilibrium $E^*$ only exists in this case. Moreover, the positive invariance of the set $S_3$ (Proposition \ref{Prop:Invariance}, 1.) proves that $V_1, V_2$ and $V_3$ are well defined when applied to the semiflow $\Phi$.
\end{proof}
We remind the definition of a Lyapunov function for the semiflow $\Phi$ in the case of infinite dimensional systems (see e.g \cite{Kuang93} Definition 5.1, p. 30 or \cite{Smith2010}, p. 80).

\begin{definition}
Let $D\subset \X$. We say that $L:\X\to \R$ is a Lyapunov function on $D$ if the following hold:
\begin{enumerate}
\item $L$ is continuous on $\overline{D}$ (the closure of $D$);	
\item $L$ decreases along orbits starting in $D$, \textit{i.e.} $t\mapsto L(\Phi_t(z))$ is a nonincreasing function of $t\geq \tau$, for every $z\in D$.
\end{enumerate}
\end{definition}

\begin{proposition}\label{Prop:Energ}
For every $z\in S_3$, the positive function
\begin{equation}\label{Eq:Energ}F_{z}:[\tau,\infty)\ni t\longmapsto L_*(\Phi_t(z))\in \R_+
\end{equation}
defined by
\begin{equation*}
F_{z}(t):=\alpha X^* g\left(\dfrac{X(t)}{X^*}\right)+y^*g\left(\dfrac{y(t)}{y^*}\right)+\alpha \beta_0 e^{-\mu_0\tau}X^* \displaystyle \int_0^\tau g\left(\dfrac{X(t+s-\tau)}{X^*}\right)ds
\end{equation*}
is nonincreasing. 
\end{proposition}

\begin{proof}
Let $z:=(\phi,y_\tau)\in S_3$. We can calculate the derivative of $L_*$:
\begin{equation*}
\begin{array}{lcl}
& &\dfrac{\partial}{\partial t}[L_*(\Phi_t(z)] \vspace{0.1cm} \\
&=&\dfrac{\partial}{\partial t}\left[\alpha X^* g\left(\dfrac{X(t)}{X^*}\right)+y^*g\left(\dfrac{y(t)}{y^*}\right)+V_3(\Phi_t(\phi,y_0))\right]\\
&=&\alpha\left(1-\dfrac{X^*}{X(t)}\right)X'(t)+\left(1-\dfrac{y^*}{y(t)}\right)y'(t)+\dfrac{\partial}{\partial t}[V_3(\Phi_t(\phi,y_\tau))].
\end{array}
\end{equation*}
We see that
$$\dfrac{d}{dt}\left[g\left(\dfrac{X(t+s)}{X^*}\right)\right]=\dfrac{d}{ds}\left[g\left(\dfrac{X(t+s)}{X^*}\right)\right]$$
so
\begin{equation*}
\begin{array}{lcl}
\dfrac{\partial}{\partial t}[V_3(\Phi_t(\phi,y_\tau))]&=&\alpha \beta_0 e^{-\mu_0 \tau}X^* \displaystyle\int_0^\tau \dfrac{d}{dt}\left[g\left(\dfrac{X(t+s-\tau)}{X^*}\right)\right]ds \vspace{0.1cm} \\
&=&\alpha \beta_0 e^{-\mu_0 \tau}X^* \displaystyle\int_0^\tau \dfrac{d}{ds}\left[g\left(\dfrac{X(t+s-\tau)}{X^*}\right)\right]ds \vspace{0.1cm} \\
&=&\alpha \beta_0 e^{-\mu_0 \tau}X^* \left[g\left(\dfrac{X(t)}{X^*}\right)-g\left(\dfrac{X(t-\tau)}{X^*}\right)\right].
\end{array}
\end{equation*}
Consequently we have
\begin{equation*}
\begin{array}{rcl}
& &\dfrac{\partial}{\partial t}[L_*(\Phi_t(z)] \vspace{0.1cm} \\

&=&\alpha\left(1-\dfrac{X^*}{X(t)}\right)\left[\beta_0 e^{-\mu_0 \tau} X(t-\tau)-\mu_0 X(t)-\gamma_0 y(t)X(t)\right]\\

&\quad& +\alpha \beta_0 e^{-\mu_0 \tau}X^*\left[g\left(\dfrac{X(t)}{X^*}\right)-g\left(\dfrac{X(t-\tau)}{X^*}\right)\right] \\

&\quad& +\left(1-\dfrac{y^*}{y(t)}\right)\left[\alpha \gamma_0 X(t)y(t)-\delta y(t)\right] \\

&=& \alpha \left(1-\dfrac{X^*}{X(t)}\right)\left[\beta_0 e^{-\mu_0 \tau} X(t-\tau) -\mu_0 X(t)\right]+ \alpha\gamma_0 y(t)X^* \\

&\quad& +\alpha\beta_0 e^{-\mu_0 \tau}\left[X(t)-X^*\ln\left(\dfrac{X(t)}{X(t-\tau)}\right)-X(t-\tau)\right] \\

&\quad& -\left(1-\dfrac{y^*}{y(t)}\right)\delta y(t) -\alpha \gamma_0 X(t) y^*
\end{array}
\end{equation*}
We know from \eqref{Eq:System} the following properties about the equilibrium:
\begin{enumerate}
\item $\alpha \gamma_0 X^*=\delta$,
\item $\alpha \mu_0 X^*+\delta y^*=\alpha \beta_0 X^* e^{-\mu_0 \tau}$,
\item $\alpha\mu_0+\alpha\gamma_0 y^*=\alpha\beta_0 e^{-\mu_0 \tau}$.
\end{enumerate}
Thus we get
\begin{equation*}
\begin{array}{rcl}
& &\dfrac{\partial}{\partial t}[L_*(\Phi_t(z))] \vspace{0.1cm} \\
&=&\alpha\left(1-\dfrac{X^*}{X(t)}\right)[\beta_0 e^{-\mu_0 \tau}X(t-\tau)-\mu_0 X(t)]+\delta y^* -\alpha \gamma_0 X(t)y^*\\

&\quad& +\alpha \beta_0e^{-\mu_0 \tau}\left[X(t)-X^*\ln\left(\dfrac{X(t)}{X(t-\tau)}\right)-X(t-\tau))\right] \\

&=& \alpha \left(1-\dfrac{X^*}{X(t)}\right)[\beta_0 e^{-\mu_0 \tau}X(t-\tau)]-\alpha\mu_0 X(t) +\alpha \beta_0 X^*e^{-\mu_0 \tau}\\

&\quad& -\alpha \gamma_0 X(t)y^*+\alpha \beta_0 e^{-\mu_0 \tau}\left[ X(t)-X^*\ln\left(\dfrac{X(t)}{X(t-\tau)}\right)-X(t-\tau))\right] \\

&=&\alpha\left(1-\dfrac{X^*}{X(t)}\right)[\beta_0 e^{-\mu_0 \tau}X(t-\tau)]-\alpha \beta_0 e^{-\mu_0 \tau} X(t)+\alpha \beta_0 X^* e^{-\mu_0 \tau}\\

&\quad& +\alpha \beta_0 e^{-\mu_0 \tau}\left[X(t)-X^*\ln\left(\dfrac{X(t)}{X(t-\tau)}\right)-X(t-\tau))\right] \\
&=& -\left(\alpha X^* \beta_0 e^{-\mu_0 \tau}\right)\left(\dfrac{X(t-\tau)}{X(t)}\right)+\alpha \beta_0 X^* e^{-\mu_0 \tau} -\alpha \beta_0 e^{-\mu_0 \tau}X^*\ln\left(\dfrac{X(t)}{X(t-\tau)}\right) 
\end{array}
\end{equation*}
Hence we obtain
$$\dfrac{\partial}{\partial t}[L_*(\Phi_t(z))] =-\alpha \beta_0 X^* e^{-\mu_0 \tau} \left(\dfrac{X(t-\tau)}{X(t)}-1+\ln\left(\dfrac{X(t)}{X(t-\tau)}\right)\right)$$
and consequently
\begin{equation}\label{Eq:Lyapunov}
\dfrac{\partial}{\partial t}[L_*(\Phi_t(z)]=-\alpha \beta_0 X^* e^{-\mu_0 \tau} g\left(\dfrac{X(t-\tau)}{X(t)}\right), \quad \forall t\geq \tau
\end{equation}
and the nonnegativity of $g$ implies that $F_z$ is a nonincreasing function.
\end{proof}
One may note that $\overline{S_3}=\X_+$. Consequently, $L_*$ cannot be a Lyapunov function on $S_3$, since it is not continuous on $\X_+\setminus S_3$ (the function explodes at the boundary, due to \eqref{Eq:Expl-Lyap}). To avoid this problem, we define for every $\varepsilon>0$, the set
$$S^\varepsilon_3:=\{(\phi,y)\in \X_+: y\geq\varepsilon, \ \phi(a)\geq\varepsilon \quad \forall a\in[0,\tau]\}\subset S_3$$
that is a closed subset of $\X_+$. We now can give the main result of this section.

\begin{corollary}\label{Cor:Lyapunov}
For every $\varepsilon>0$, $L_*$ is a Lyapunov function on $S^\varepsilon_3$.
\end{corollary}

\begin{remark}
Note that, to perform the global asymptotic analysis of the extinction equilibrium $\overline{E_0}$, one could use the functional: 
$$L_0(\phi,y)=\alpha \phi(\tau)+y+\alpha\beta_0 e^{-\mu_0 \tau}\int_{0}^\tau \phi(s)\mathrm{d}s$$
formally defined for $(\phi,y)\in \X_+$. Then one can deduce the global stability of $E_0$ in $\X_+$ when $R_0<1$. This result was already obtained in \cite{PerassoRichard19}  Theorem 3.5, without the use of Lyapunov function.
\end{remark}

\subsection{Attractive set of the solutions}\label{Sec:Attractive_Set}

We start by proving the boundedness of the solutions.
\begin{lemma}\label{Lemma:Bounded}
For every $z\in S_2$, there exists a finite constant $C(z)>0$, such that $X(t)\leq C(z)$ and $y(t)\leq C(z)$, for every $t\geq \tau$.
\end{lemma}

\begin{proof}
Let $z:=(\phi,y_\tau)\in S_3$. Consequently to Proposition \ref{Prop:Energ}, for every $t\geq \tau$, we have $F_{z}(t)\leq F_{z}(\tau)$, where
$$F_{z}(\tau)=\alpha X^*g\left(\dfrac{\phi(\tau)}{X^*}\right)+y^*g\left(\dfrac{y_\tau}{y^*}\right)+\alpha \beta_0 e^{-\mu_0\tau}X^*\int_0^\tau g\left(\dfrac{\phi(s)}{X^*}\right)ds.$$
Since each term of $F_{z}$ is positive and
$$\lim_{x\to \infty}g(x)=\infty,$$
then there exists a positive constant $C(z)>0$ such that
$$X(t)\leq C(z), \quad y(t)\leq C(z), \quad \forall t\geq \tau.$$

Now suppose that $z\in S_2$. By means of Remark \ref{Remark:Explosion}, the solution $(X(t),y(t))$ is bounded for $t\in[\tau,3\tau]$. Moreover, from Proposition \ref{Prop:Invariance}, 2) and the fact that $\Phi_t(z)\in S_3$ for every $t\geq 3\tau$, we deduce that $(X(t),y(t))$ is bounded for any $t\geq \tau$.
\end{proof}
We continue with a persistence result.

\begin{lemma}\label{Lemma:Pers}
For every $z\in S_3$, there exists $\varepsilon>0$ such that
$$\Phi_t(z)\in S^\varepsilon_3, \quad \forall t\geq \tau.$$
\end{lemma}

\begin{proof}
Let $z\in S_3$. Suppose by contradiction that for every $\varepsilon>0$ there exists $t\geq \tau$ such that
$$X(t)<\varepsilon \quad \text{or} \quad y(t)<\varepsilon.$$
Letting $\varepsilon$ go to $0$ implies that $L_*(\Phi_t(z))$ goes to infinity leading to a contradiction with Proposition \ref{Prop:Energ}.
\end{proof}
In all the following, any `$\tau$-periodic function' will be not constant. We are now ready to compute the attractive set of the solutions.

\begin{theorem}\label{Thm:PeriodConv}
For every initial condition $z\in S_2$, the solution $(X,y)$ converge either to a $\tau$-periodic function or to $E^*$.
\end{theorem}

\begin{proof}
First, consider an initial condition $z\in S_3$. By Lemma \ref{Lemma:Pers}, there exists $\varepsilon>0$ such that
$$\Phi_t(z) \in S^\varepsilon_3, \quad \forall t\geq \tau.$$
Using Corollary \ref{Cor:Lyapunov} and Lemma \ref{Lemma:Bounded}, we know that $L_*$ is a Lyapunov function on $S^\varepsilon_3$ and $\Phi_t(z)$ is a bounded solution. Consequently to LaSalle invariance principle (see \cite[Theorem 5.3, p. 30]{Kuang93} or \cite[Theorem 5.17, p. 80]{Smith2010}), we conclude that $\omega(z)\neq\emptyset$ and is contained in the maximal invariant subset of
$$\left\{v\in S^\varepsilon_3: \dfrac{\partial}{\partial t}[L_*(\Phi_t(v))]=0, \quad \forall t\geq \tau\right\}.$$
We see that \eqref{Eq:Lyapunov} implies
\begin{equation}\label{Eq:Period}
X(t-\tau)=X(t), \quad \forall t\geq \tau,
\end{equation}
so $\omega(z)$ is included in
$$\{v\in S^\varepsilon_3: \Phi^X_t(v)=\Phi^X_{t+\tau}(v), \quad \forall t\geq \tau\},$$
where $\Phi^X$ is the first component of $\Phi$. Classical results (see e.g. \cite{Hale93}) imply that $X\in \Co ^1[\tau,\infty)$. Therefore we get
$$X'(t)=X'(t+\tau), \qquad \forall t\geq \tau,$$
which implies
$$\gamma_0 X(t)y(t)=\gamma_0 X(t+\tau)y(t+\tau), \qquad \forall t\geq \tau$$
hence
\begin{equation}\label{Eq:Conv_y}
y(t)=y(t+\tau), \qquad \forall t\geq \tau.
\end{equation}
Suppose that
$$X(t)=c\in (0,\infty) \quad \textnormal{for all} \ t\geq \tau.$$
Then \eqref{Eq:System} implies that
$$y(t)=y^* \quad \textnormal{for all} \ t\geq \tau$$
which leads to $c=X^*$, whence $\omega(z)=\{\overline{E^*}\}$ in this case.
Now, suppose that $X$ is a $\tau$-periodic function. Suppose also, by contradiction, that $y$ is not a $\tau$-periodic function. Using \eqref{Eq:Conv_y}, we would obtain that $y$ is constant and then that
$$X(t)=X^* \quad \text{for all} \ t\geq \tau$$
due to \eqref{Eq:System} which leads to a contradiction. Hence the result follows. \\
Now, consider an initial condition $z\in S_2$. Using Proposition \ref{Prop:Invariance} 2), we know that $\Phi_{3\tau}(z)\in S_3$. We can therefore use the proof above to get the same asymptotic result.
\end{proof}

\subsection{Existence of a $\tau$-periodic solution}

By means of the latter result, the convergence to a $\tau$-periodic function is a possible case. We now give a necessary and sufficient condition to get the existence of such periodic solution.

\begin{theorem}\label{Thm:NoPeriod}
There exists a $\tau$-periodic solution of \eqref{Eq:System} if and only if
\begin{equation}
\label{Eq:CondPeriod}
\dfrac{\tau\sqrt{\delta y^*\gamma_0}}{2\pi}>1
\end{equation}
holds. In this case, the solution is unique (in the sense that there is only one $\tau$-periodic orbit) and will be denoted by $(p,q)\in \Co ^1([\tau,\infty),\R^2_+)$ in all the following.
\end{theorem}
Let us first remind some useful property about the classical Lotka-Volterra model.

\begin{lemma}\cite[Theorem 1]{Rothe85} \label{Lemma:Period}
The solution of 
\begin{equation}\label{Eq:LVmodel}
\left\{
\begin{array}{rcl}
x'(t)&=&ax(t)-bx(t)y(t), \\
y'(t)&=&cx(t)y(t)-dy(t), \\
(x(\tau),y(\tau))&=&(x_\tau,y_\tau)\in (0,\infty)^2,
\end{array}
\right.
\end{equation}
for every $t\geq \tau$, is periodic with some period $T$. Define the conserved energy $\E_{(x_\tau,y_\tau)}$ (through time) of \eqref{Eq:LVmodel} by
\begin{eqnarray}\label{Eq:E}
\E_{(x_\tau,y_\tau)}&=&cx_\tau-d+by_\tau-a-a\ln\left(\dfrac{by_\tau}{a}\right)-d\ln\left(\dfrac{cx_\tau}{d}\right)\nonumber \\
&=&dg\left(\dfrac{cx_\tau}{d}\right)+ag\left(\dfrac{by_\tau}{a}\right)\geq 0,
\end{eqnarray}
which depends on the initial condition. Then the period depends on $\E_{(x_\tau,y_\tau)}$ and moreover, the function $\E\mapsto T(\E)$ is strictly increasing with
$$\lim_{\E\to 0}T(\E)=\dfrac{2\pi}{\sqrt{ad}}, \qquad \lim_{\E\to \infty}T(\E)=\infty.$$
\end{lemma}

\begin{proof}(Theorem \ref{Thm:NoPeriod}.)
If $(X,y)$ is a $\tau$-periodic solution of \eqref{Eq:System}, then it is actually solution of 
\begin{equation}\label{Eq:SystemSimp}
\left\{
\begin{array}{rcl}
X'(t)&=&\left(\beta_0 e^{-\mu_0 \tau}-\mu_0\right)X(t)-\gamma_0 X(t)y(t), \\
y'(t)&=& \alpha \gamma_0 X(t)y(t)-\delta y(t),
\end{array}
\right.
\end{equation}
for any $t\geq \tau$. Suppose that
$$\dfrac{\tau\sqrt{\delta y^*\gamma_0}}{2\pi}<1.$$
Using Lemma \ref{Lemma:Period}, for each initial condition, the solution is periodic with some period $T$. Since the period is strictly increasing, it must satisfy
$$T\geq \dfrac{2\pi}{\sqrt{(\beta_0 e^{-\mu_0\tau}-\mu_0)\delta}}=\dfrac{2\pi}{\sqrt{\delta \gamma_0 y^*}}>\tau,$$
which is absurd. If
$$\dfrac{\tau\sqrt{\delta y^*\gamma_0}}{2\pi}=1,$$
then to get $T=\tau$, one needs to have $\E_{(x_\tau,y_\tau)}=0$. Using \eqref{Eq:E}, we get
$$x_\tau=\dfrac{d}{c}, \quad y_\tau=\dfrac{a}{b}$$
which is equivalent, for \eqref{Eq:SystemSimp}, to
$$x_\tau=X^*, \quad y_\tau=y^*,$$
so the solution is actually constant and the first implication is thus proved.

Conversely, suppose that $\eqref{Eq:CondPeriod}$ is satisfied. Using Lemma \ref{Lemma:Period}, there is a unique energy $\E^*>0$ such that
$$T(\E^*)=\tau.$$
Moreover, using \eqref{Eq:E}, we can see that there is at least one initial condition $(x_1,y_1)\in \R^2$ such that
$$\mathcal{E}_{(x_1,y_1)}=\E^*.$$
Thus, there is at least one $\tau$-periodic solution of \eqref{Eq:SystemSimp} (denoted by $(p,q)$). Besides, every initial condition $(x_2,y_2)$ that satisfies
$$\E_{(x_2,y_2)}=\E^*$$
belongs to 
$$\bigcup_{s\in[\tau,2\tau]}\{(p(s),q(s))\}.$$ Consequently, there is a unique $\tau$-periodic solution $(p,q)\in \Co^1([\tau,\infty), \R_+^2)$ up to a phase shift, of \eqref{Eq:SystemSimp}. We finally see that $(p,q)$ is also solution of \eqref{Eq:System}, which ends the proof.
\end{proof}

\begin{remark} We can note that the $\tau$-periodic solution of \eqref{Eq:System} is linked with the existence of two purely imaginary roots (given by \eqref{Eq:Lambda}) of the characteristic equation \eqref{Eq:LocalStab}, whenever 
$$\dfrac{\tau\sqrt{\delta y^* \gamma_0}}{2\pi}=1$$
holds. Indeed, in this case, the condition \eqref{Eq:Cond_Imag} is satisfied.\end{remark}

We can now be more precise about the attractive set of the solutions.

\begin{proposition}\label{Prop:Attract-Set}
Consider an initial condition $z\in S_2$. 
\begin{enumerate}
\item If \eqref{Eq:CondPeriod} does not hold, then
\begin{equation}\label{Eq:Attrac_NoPer}
\omega(z)=\{\overline{E^*}\}.
\end{equation}
\item If \eqref{Eq:CondPeriod} holds then
\begin{equation}\label{Eq:Attrac_Per}
\omega(z)\subset \{\overline{E^*}\}\cup S_\tau,
\end{equation}
\end{enumerate}
where $S_\tau\subset S_3$ is the (periodic) positively invariant subset of $S_3$ defined by
\begin{equation}\label{Eq:S_tau}
\begin{array}{rcl}
&& S_\tau:=\left\{(\phi,y)\in \X_+ : \ \exists h\in[0,\tau], \phi(\cdot)=p(\cdot+\tau+h), \right. \left. y=q(\tau+h)\right\}.
\end{array}
\end{equation}
\end{proposition}

\begin{proof}
The result follows from Theorem \ref{Thm:PeriodConv} and Theorem \ref{Thm:NoPeriod}.
\end{proof}

\subsection{Lyapunov stability}
Here we handle the behavior of the solutions around the non trivial equilibrium. We know by Theorem \ref{Thm:LocStab} that $E^*$ is locally asymptotically stable when \eqref{Eq:Cond_Imag} does not hold. We can now be more precise:
\begin{proposition}\label{Prop:Stability}
The equilibrium $E^*$ is Lyapunov stable.
\end{proposition}
To prove this result, we need to define the following sets
$$L_\eta=\{(\phi,y)\in \X_+: L_*(\phi,y)<\eta\}, \quad \eta >0,$$
$$B(E^*,\rho)=\{(w,y)\in \R^2: \|(w,y)-E^*\|_{\R^2}\leq \rho\}, \quad \rho>0,$$
where $\|(w,y)\|_{\R^2}=|w|+|y|$, for any $(w,y)\in \R^2$;
$$B(\overline{E^*},\rho)=\{(\phi,y)\in \X_+: \|(\phi,y)-\overline{E^*}\|_{\X}\leq \rho\}, \quad \rho>0,$$
and we give two lemmas (see \cite[Proof of Theorem 1.2]{Gabriel2012} for the idea of such results).

\begin{lemma}\label{Lemma:LevelSet_1}
For every $\rho>0$, there exists $\eta>0$ such that $(\phi,y)\in L_\eta \Rightarrow (\phi(\tau),y)\in B(E^*,\rho)$.
\end{lemma}

\begin{proof}
Let $\rho>0, \eta>0$ and $(\phi_\eta,y_\eta)\in L_\eta$. We have $L_*(\phi_\eta,y_\eta)<\eta$ so
$$V_1(\phi_\eta, y_\eta)<\eta, \qquad V_2(\phi_\eta, y_\eta)<\eta,$$
and
$$g\left(\dfrac{\phi_\eta(\tau)}{X^*}\right)<\dfrac{\eta}{\alpha X^*}, \qquad g\left(\dfrac{y_\eta}{y^*}\right)<\dfrac{\eta}{y^*}.$$
Since $g$ is nonnegative then
$$\lim_{\eta \to 0} \ g\left(\dfrac{\phi_\eta(\tau)}{X^*}\right)=0, \qquad \lim_{\eta \to 0} \ g\left(\dfrac{y_\eta}{y^*}\right)=0,$$
and, since $g$ is zero only at $1$, we obtain
$$\lim_{\eta\to 0} \ \phi_\eta(\tau)=X^*, \qquad \lim_{\eta \to 0} \ y_\eta =y^*.$$
By considering $\eta>0$ small enough we get $\|(\phi_\eta(\tau),y_\eta)-E^*\|_{\R^2}\leq \rho$ and $(\phi_\eta(\tau),y_\eta)\in B(E^*,\rho)$.
\end{proof}

\begin{lemma}\label{Lemma:LevelSet_2}
For every $\eta>0$, there exists $\rho>0$ such that $ B(\overline{E^*},\rho)\subset L_\eta$.
\end{lemma}

\begin{proof}
Let $\eta>0, \rho>0$ and $(\phi_\rho,y_\rho)\in B(\overline{E^*}, \rho)$, then
$\|(\phi_\rho,y_\rho)-\overline{E^*}\|_\X\leq \rho$ so we get
$$\|\phi_\rho-X^*\mathbf{1}_{[0,\tau]}\|_\infty\leq \rho, \qquad |y_\rho-y^*|\leq \rho.$$
Consequently we have
$$\lim_{\rho \to 0}y_\rho=y^*, \qquad \lim_{\rho \to 0}\phi_\rho(s)=X^*, \ \forall s\in[0,\tau],$$
and then
$$\lim_{\rho \to 0} \ g\left(\dfrac{y_\rho}{y^*}\right)=0, \quad \lim_{\rho \to 0} \ g\left(\dfrac{\phi_\rho(s)}{X^*}\right)=0, \ \forall s\in[0,\tau].$$
Consequently 
$$\lim_{\rho \to 0}V_1(\phi_\rho,y_\rho)=0, \quad \lim_{\rho \to 0}V_2(\phi_\rho,y_\rho)=0, \quad \lim_{\rho \to 0}V_3(\phi_\rho,y_\rho)=0.$$
So, considering $\rho>0$ small enough, we get $L_*(\phi_\rho,y_\rho)\leq \eta$.
\end{proof}

\begin{proof}(Proposition \ref{Prop:Stability}.)
Let $\rho_1>0$. Using Lemma \ref{Lemma:LevelSet_1}, there exists $\eta>0$ such that
$$(\phi,y)\in L_\eta \Rightarrow (\phi(\tau),y)\in B(E^*,\rho_1)$$
and using Lemma \ref{Lemma:LevelSet_2}, there exists $\rho_2>0$ such that 
$$B(\overline{E^*},\rho_2)\subset L_\eta.$$
Let $(\phi,y)\in B(\overline{E^*}, \rho_2)$, then $(\phi,y)\in L_\eta$ so $(\phi(\tau),y)\in B(E^*,\rho_1)$. Since $F_{(\phi,y)}$ is nonincreasing, then $L_\eta$ is positively invariant, which implies
$$(\Phi^X_t(\phi,y)(\tau),\Phi^y_t(\phi,y))\in B(E^*, \rho_1), \ \forall t\geq \tau$$
where $\Phi^y$ is the second component of $\Phi$, so that
$$(X(t),y(t))\in B(E^*,\rho_1), \ \forall t\geq \tau.$$
Consequently
$$|X(t)-X^*|+|y(t)-y^*|\leq \rho_1, \ \forall t\geq \tau.$$
Since $(\phi,y)\in B(\overline{E^*},\rho_2)$, then we have
$$\|\phi-X^*\mathbf{1}_{[0,\tau]}\|_\infty+|y-y^*|\leq \rho_2.$$
Considering $\rho_2>0$ small enough, that satisfies $\rho_2\leq \rho_1$, leads to
$$\|X_t-X^*\mathbf{1}_{[0,\tau]}\|_\infty+|y(t)-y^*|\leq \rho_1, \ \forall t\geq \tau$$
that is
$$\|(X_t,y(t))-\overline{E^*}\|_\X\leq \rho_1, \ \forall t\geq \tau$$
so
$$\Phi_t(\phi,y)\in B(\overline{E^*},\rho_1), \ \forall t\geq \tau.$$
We have finally shown that $E^*$ is Lyapunov stable, since for every $\rho_1>0$ there exists $\rho_2>0$ such that 
$$(\phi,y)\in B(\overline{E^*},\rho_2) \Rightarrow \Phi_t(\phi,y)\in B(\overline{E^*},\rho_1), \ \forall t\geq \tau.$$
\end{proof}

\subsection{Asymptotic behavior in absence of periodic solution}

In absence of $\tau$-periodic solution,  i.e. \eqref{Eq:CondPeriod} does not hold, the behavior of the solutions is given by the following theorem:

\begin{theorem}\label{Thm:Conv_NoPeriod}
If \eqref{Eq:CondPeriod} is not satisfied, then $E^*$ is globally asymptotically stable in $S_2$.
\end{theorem}

\begin{proof}
We know that \eqref{Eq:Attrac_NoPer} holds for every $z\in S_2$. Consequently, the global stability of $\overline{E^*}$ (and $E^*$) in the basin $S_2$, when \eqref{Eq:CondPeriod} does not hold, is just a consequence of Proposition \ref{Prop:Stability}.
\end{proof}

\subsection{Asymptotic behavior in presence of a periodic solution}

Let us suppose now that there exists a $\tau$-periodic solution, i.e. that \eqref{Eq:CondPeriod} holds. In this case, we already know that
$$\omega(z)\subset \{\overline{E^*}\}\cup S_{\tau}, \qquad \forall z\in S_2$$
where $S_{\tau}$ is defined by \eqref{Eq:S_tau}. We start by proving the global asymptotic stability of $\overline{E^*}$ in a subset of $S_2\setminus S_{\tau}$.

\begin{remark}\label{Rem:Eps_tau}
Using Theorem \ref{Thm:NoPeriod}, we know that, in this case, there is a unique nonconstant $\tau$-periodic solution, up to a phase shift, $(p,q)\in \Co ^1([\tau,\infty),\R^2_+)$ for \eqref{Eq:System}. Let $(\overline{p},\overline{q})\in \X_+$ be defined by
$$\overline{p}(s)=p(s+\tau), \quad \forall s\in[0,\tau], \quad \overline{q}=q(\tau).$$
It is clear that $(\overline{p}, \overline{q})\in S_{\tau}$ and that $\Phi_t(\overline{p}, \overline{q})\in S_{\tau}$, $\forall t\geq \tau$. Moreover the following equivalence holds true, by \eqref{Eq:S_tau}:
$$(\phi,y)\in S_{\tau}\Longleftrightarrow \exists \ h\in[\tau,2\tau]: \Phi_h(\phi,y)=(\overline{p}, \overline{q}).$$
We then define the (constant) energy for the periodic function by $\E_\tau :=F_{(\overline{p},\overline{q})}(\tau)$, i.e.
\begin{equation*}
\E_\tau=\alpha X^*g\left(\dfrac{\overline{p}(\tau)}{X^*}\right)+y^*g\left(\dfrac{\overline{q}}{y^*}\right)+\alpha \beta_0 e^{-\mu_0 \tau}X^* \int_0^\tau g\left(\dfrac{\overline{p}(s)}{X^*}\right)ds\in(0,\infty)
\end{equation*}
and we deduce that
$$S_{\tau}\subset \{z\in S_2: L_*(z)=\E_{\tau}\}.$$
\end{remark}

\begin{proposition}\label{Prop:GAS_part}
If \eqref{Eq:CondPeriod} is satisfied, then $E^*$ is globally asymptotically stable in
$$\{z\in S_2 : L_*(z)\leq \E_\tau\}\setminus S_\tau.$$
\end{proposition}
\begin{proof}
Since $E^*$ is stable by Proposition \ref{Prop:Stability}, it remains to prove the attractiveness. We see that
$$L_{\E_\tau} \subset \{z\in S_2 : L_*(z)\leq \E_\tau\}\setminus S_\tau.$$
First, let $z\in L_{\E_\tau}$ and define
$$\E_*:=F_z(\tau)<\E_\tau.$$
We know that \eqref{Eq:Attrac_Per} holds. If $\omega(z)\subset S_\tau$, then there would exist a time $t^*\geq \tau$ such that
$$\E_*<F_z(t^*)<\E_\tau$$
which contradict the fact that $F_z$ is nonincreasing. Consequently \eqref{Eq:Attrac_NoPer} actually holds.
Now, let
$$z\in \{w\in S_2 : L_*(w)=\E_\tau\}\setminus S_\tau,$$
and suppose that $\omega(z)\subset S_\tau$. Then one needs to have
$$F_z(t)=\E_\tau, \quad \forall t\geq \tau,$$
i.e. $F_z$ must be constant. Using Equation \eqref{Eq:Lyapunov}, it implies that 
$$X(t-\tau)=X(t), \quad \forall t\geq \tau.$$
As in the proof of Theorem \ref{Thm:PeriodConv}, either $(X(t),y(t))\equiv (X^*,y^*)$ but then we would have $L_*(z)=L_*(\overline{E^*})=0$ that is absurd, or $(X,y)$ is a $\tau$-periodic function, which is also absurd since $$z\not\in S_\tau.$$
Consequently \eqref{Eq:Attrac_NoPer} holds and the asymptotic stability follows.
\end{proof}
We can deduce the following result, without supposing \eqref{Eq:Cond_Imag} as we did for Theorem \ref{Thm:LocStab}
\begin{corollary}\label{Cor:LAS}
The nontrivial equilibrium $E^*$ is locally asymptotically stable.
\end{corollary}

\begin{proof}
Since
$$\mathcal{E}_{\overline{E^*}}=F_{\overline{E^*}}(\tau)=0,$$
then, by continuity of $L^*$, we can find a neighborhood of $\overline{E^*}$, denoted by $V_{\overline{E^*}}$, such that
$$V_{\overline{E^*}}\subset \{z\in S_2: L_*(z)\leq \mathcal{E}_\tau\}\setminus S_\tau.$$
Consequently, for every initial condition $z\in V_{\overline{E^*}}$, the solution of \eqref{Eq:System} will converge to $E^*$, whence the local asymptotic stability.
\end{proof}
We now focus on the $\tau$-periodic solution by proving its unattractiveness.

\begin{definition}
Let $S\subset \X$ be a subset of $\X$. We say that $(p,q)$ is \textbf{weakly orbitally unattractive} in $S$ if, for every $\eta>0$, there exist $h\in[0,\tau]$ and $(\phi,y_\tau)\in S$ that satisfies $\|(\phi,y_\tau)-(\overline{p},\overline{q})\|_{\X}\leq \eta$ such that
\begin{equation}\label{Eq:Unattrac}
\limsup_{t\to \infty}\|(\Phi^X_{t+h}(\phi,y_\tau)(\tau),\Phi^y_{t+h}(\phi,y_\tau))-(p,q)(t)\|_{\R^2}>0;
\end{equation}
%\textbf{strongly orbitally unattractive} in $S$ if, for every initial condition $(\phi,y_0)\in S\setminus S_\tau$, there exists $h\in[0,\tau]$ such that \eqref{Eq:Unattrac} is satisfied.
\end{definition}
\begin{remark}Note that $(\Phi^X_{t+h}(\phi,y_\tau)(\tau), \Phi^y_{t+h}(\phi,y_\tau))$ is simply the solution of \eqref{Eq:System} with the initial condition $(\phi, y_\tau)$ at time $t+h$, \textit{i.e.} $(X(t+h),y(t+h)).$
\end{remark}
We need:
\begin{lemma}
One can suppose without loss of generality, that
$$p(\tau)=X^*, \qquad q(\tau)\neq y^*,$$
so that $\overline{p}(0)=\overline{p}(\tau)=X^*$, $\overline{q}\neq y^*$ and
\begin{equation}\label{Eq:EnergPeriod}
\E_\tau=y^*g\left(\dfrac{\overline{q}}{y^*}\right)+\alpha \beta_0 e^{-\mu_0 \tau}X^* \int_0^\tau g\left(\dfrac{\overline{p}(s)}{X^*}\right)ds.
\end{equation}
\end{lemma}

\begin{proof}
There necessarily exists $t^*\in[\tau,2\tau]$ such that $p(t^*)=X^*$. Indeed, if
$$p(t)<X^*, \ \forall t\in[\tau,2\tau],$$
then
$$q'(t)<0, \ \forall t\in[\tau,2\tau],$$
so $q$ is decreasing on $[\tau,2\tau]$ and cannot be $\tau$-periodic. Similarly, if
$$p(t)>X^*, \ \forall t\in[\tau,2\tau],$$
then $q$ would be increasing on $[\tau,2\tau]$. Let say, without loss of generality, that $t^*=\tau$. Now suppose that
$$q(\tau)=y^*.$$
Since $(p,q)$ is solution of \eqref{Eq:LVmodel} with $(x_\tau,y_\tau)=(X^*,y^*)$, we would get
$$\E_{(x_\tau,y_\tau)}=0$$
so
$$p(t)=X^*, \qquad q(t)=y^*, \qquad \forall t\geq \tau$$
which is absurd. Consequently $q(\tau)\neq y^*$.
\end{proof}
We now prove the following:
\begin{proposition}\label{Prop:Orb}
The $\tau$-periodic function $(p,q)$ is weakly orbitally unattractive in $S_2$.
\end{proposition}

\begin{proof}
We know from Proposition \ref{Prop:Energ} that for every initial condition $(\phi,y_\tau)\in S_3$, the function $F_{(\phi,y_\tau)}$ defined by \eqref{Eq:Energ} is nonincreasing. We see that  the energy for the periodic function, denoted by $\E_\tau$, is given by \eqref{Eq:EnergPeriod}. Consider
$$y_\tau=\overline{q}+\eta\left(\dfrac{y^*-\overline{q}}{|y^*-\overline{q}|}\right)$$
with $\eta>0$ small enough such that 
$$g\left(\dfrac{y_\tau}{y^*}\right)<g\left(\dfrac{\overline{q}}{y^*}\right)$$
which is possible with the fact that $\overline{q}\neq y^*$ and since $g$ is decreasing on $(0,1]$ and increasing on $[1,\infty)$). Consequently we get
$$\E_*:=F_{(\overline{p},y_\tau)}(\tau)=y^*g\left(\dfrac{y_\tau}{y^*}\right)+\alpha \beta_0 e^{-\mu_0 \tau}X^* \int_0^\tau g\left(\dfrac{\overline{p}(s)}{X^*}\right)ds<\E_\tau .$$
Recalling that \eqref{Eq:Attrac_Per} holds, we deduce that if we had $\omega(z)\subset S_\tau$ then there would exist a time $t^*>\tau$ such that
$$\E_*<F_{(\overline{p},y_0)}(t^*)<\E_\tau $$
but it would contradict the fact that the function $F_{(\overline{p},y_\tau)}$ is nonincreasing. Consequently \eqref{Eq:Attrac_NoPer} holds and \eqref{Eq:Unattrac} is satisfied. We readily see that $\eta>0$ can be taken as small as we want. The weak unattractiveness in $S_3$ is then obtained with the fact that, if it is true for small $\eta>0$, then it is clearly true for all $\eta>0$.
\end{proof}

\begin{remark}\label{Rem:Orb}
Note that the latter result and Proposition \ref{Prop:Attract-Set} induce that one can find some initial conditions, near the periodic solution, such that the solution of \eqref{Eq:System} converges to $E^*$. The question whether the unattractiveness is strong (i.e. true for every initial conditions in $S_2\setminus S_{\tau}$) is an open problem.
\end{remark}

\subsection{Numerical simulations}

In this section, we show some numerical simulations to illustrate the results proven above. We consider $\mu_0=0.5, \tau=3, \gamma_0=0.5, \alpha=0.7, \delta=2$ and we let $\beta_0$ vary. If $\beta_0=10,$ then  \eqref{Eq:CondPeriod} does not hold (the value is around $0.89$) and consequently to Theorem \ref{Thm:Conv_NoPeriod}, we get the convergence to $E^*$ whatever the initial condition taken in $S_2$ (see Figure \ref{Figure1}). Now, if $\beta_0=20$, then \eqref{Eq:CondPeriod} holds (the value is around $1.34$). In one hand, Proposition \ref{Prop:GAS_part} implies the convergence to $E^*$ in a subset of $S_2$ (see Figure \ref{Figure2} for two different sets of initial conditions). On the other hand, Proposition \ref{Prop:Orb} implies that $S_\tau$ is weakly orbitally unattractive, and by Remark \ref{Rem:Orb} we know that there exists some initial conditions near the periodic solution, the solution of \eqref{Eq:System} converges to $E^*$ (see Figure \ref{Figure3}). All these simulations let us think that when \eqref{Eq:CondPeriod} holds, the equilibrium $E^*$ is globally asymptotically stable in $S_2\setminus S_\tau$ and that the $\tau$-periodic solution is strongly unattractive in $S_2\setminus S_\tau$.

\begin{figure}[hbtp]
\begin{center}
\includegraphics[scale=0.3]{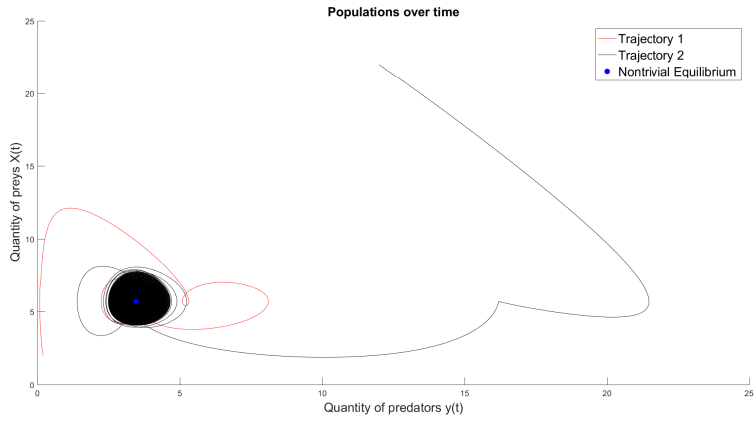}
\caption{When \eqref{Eq:CondPeriod} does not hold}
\label{Figure1}
\end{center}
\end{figure}

\begin{figure}[hbtp]
\begin{center}
\includegraphics[scale=0.3]{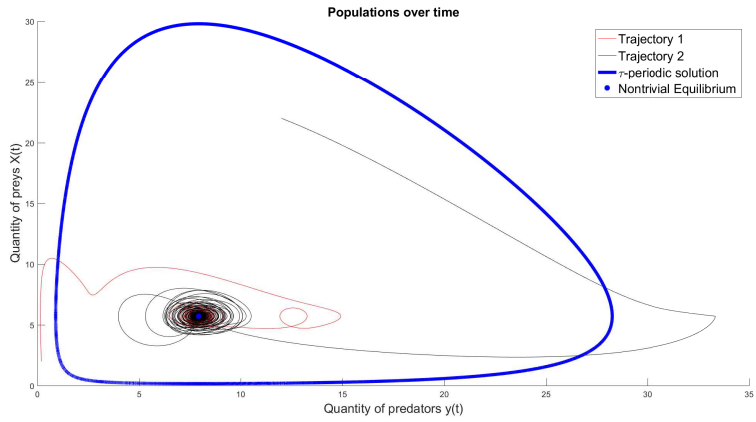}
\caption{When \eqref{Eq:CondPeriod} holds and initial conditions away from the $\tau$-periodic solution}
\label{Figure2}
\end{center}
\end{figure}

\begin{figure}[hbtp]
\begin{center}
\includegraphics[scale=0.3]{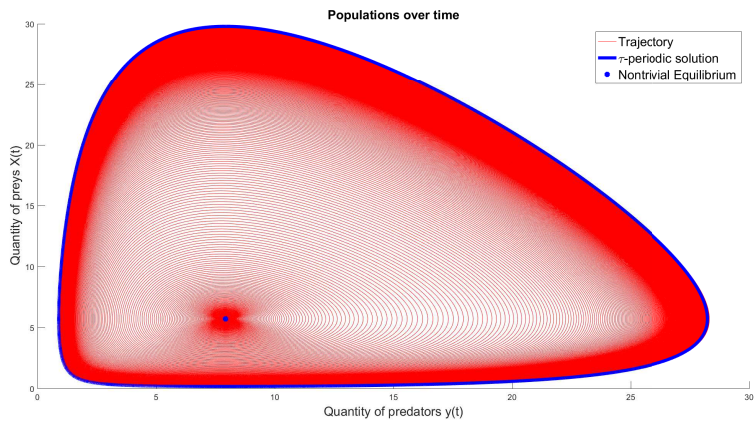}
\caption{When \eqref{Eq:CondPeriod} holds and initial condition near the $\tau$-periodic solution}
\label{Figure3}
\end{center}
\end{figure}

\section{Back to the PDE model}\label{Sec:PDE}

In this section, we return to the initial PDE predator-prey model and we prove an asymptotic stability result for the nontrivial equilibrium $E_2:=(x_2,y_2)\in \Y_+$, where $(x_2,y_2)$ satisfies the following system:
\begin{equation}\label{Eq:Sys_E2}
\left\{
\begin{array}{ccc}
x_2(a)=x_2(0)e^{-\int_0^a \mu(s)ds-y_2\int_0^a \gamma(s)ds}, \\
x_2(0)\left[1-\int_0^\infty \beta(a)e^{-\int_0^a \mu(s)ds-y_2\int_0^a \gamma(s)ds}da\right]=0, \\
y_2\left[\alpha \int_0^\infty \gamma(a)x_2(a)da-\delta\right]=0.
\end{array}
\right.
\end{equation}

\subsection{Attractiveness of $E_2$}

By analogy with the set $S_2$ for the delay problem, we define for the PDE case:
$$\Y_2:=\{(x_0,y_0)\in L^1_+(\R_+)\times (0,\infty): \int_0^\infty x_0(a)da>0\}.$$
We can prove:

\begin{theorem}\label{Thm:PDE_Att}
Suppose that \eqref{Eq:CondPeriod} does not hold. Under the assumption \eqref{Eq:PartCase}, the equilibrium $E_2$ is globally attractive in $\Y_2$ for \eqref{Eq:PDEmodel}.
\end{theorem}

\begin{proof}
Let $(x_0, y_0)\in \Y_2$ and $(x,y)$ be the solution of \eqref{Eq:PDEmodel}. We get
$$x(\tau,a)\geq x_0(a-\tau)e^{-\mu_0 \tau}e^{-\gamma_0 \tau M},$$
for every $a\geq \tau$ where $M=\max_{s\in [0,\tau]}y(s)<\infty$. Therefore 
$$\int_{\tau}^\infty x(\tau,a)da>0$$
and we also have
$$y(\tau)\geq y_0 e^{-\delta \tau}>0.$$
We can then consider for \eqref{Eq:System} the initial condition $z=(\phi, y(\tau))\in \X$, 
where
$$\phi(\theta)=\int_\tau^\infty x(\theta,a)da,$$
for every $\theta\in [0,\tau]$. Since $\phi(\tau)>0$, we can check by continuity that $\int_0^\tau \phi(s)ds>0$ whence $z\in S_2$. We know, by Theorem \ref{Thm:Conv_NoPeriod}, that $E^*$ is globally asymptotically stable in $S_2$ (for \eqref{Eq:System}). Consequently, we get
$$\lim_{t\to \infty}y(t)=y^*, \qquad \lim_{t\to \infty}X(t)=X^*$$
hence 
$$\lim_{t\to \infty}x(t,0)=\lim_{t\to \infty}\beta_0 \int_\tau^\infty x(t,a)da=\lim_{t\to \infty}\beta_0 X(t)=\beta_0 X^*.$$
Let $\varepsilon>0$, then there exists $t^*>\tau$ such that for every $t\geq t^*$, we have $|X(t)-X^*|\leq \varepsilon$ and $|y(t)-y^*|\leq \varepsilon$. The positivity of $(x,y)$, obtained in \cite[Theorem 2.3]{PerassoRichard19}, implies that for $t\geq t^*$, we have
\begin{equation*}
\left\{
\begin{array}{lcl}
\partial_a x(t,a)+\partial_t x(t,a)=-\mu_0 x(t,a)-\gamma_0 \chi_{[\tau,\infty)}(a)x(t,a),\\
\beta_0 (X^*-\varepsilon) \leq x(t,0)\leq \beta_0 (X^*+\varepsilon).
\end{array}
\right.
\end{equation*}
For every $a\leq t$, we thus get
$$\beta_0 (X^*-\varepsilon) e^{-\mu_0 a} \leq x(t,a)\leq \beta_0 (X^*+\varepsilon) e^{-\mu_0 a}$$
if $a\in[0,\tau]$, and
$$ \beta_0 (X^*-\varepsilon) e^{-\mu_0 a -\gamma_0(a-\tau)(y^*+\varepsilon)}\leq x(t,a)\leq \beta_0 (X^*+\varepsilon) e^{-\mu_0 a-\gamma_0(a-\tau)(y^*-\varepsilon)}$$
if $a\in[\tau,t]$.
Letting $\varepsilon>0$ go to $0$, we deduce that
\begin{equation*}
\lim_{t\to \infty}x(t,a)=\begin{cases}
\beta_0 X^* e^{-\mu_0 a} & \ \forall a\in[0,\tau], \\
\beta_0 X^* e^{-\mu_0 a}e^{-\gamma_0 y^* (a-\tau)} & \ \forall a\in[\tau,\infty).
\end{cases}
\end{equation*}
Since $(x_2,y_2)$ satisfy \eqref{Eq:Sys_E2}, then we see that
$$y_2=\dfrac{\beta_0 e^{-\mu_0 \tau}-\mu_0}{\gamma_0}=y^*.$$
Moreover we have
\begin{equation*}
x_2(a)=
\begin{cases}
x_2(0)e^{-\mu_0 a} &\text{if } a\in[0,\tau], \\
x_2(0)e^{-\mu_0 a -\gamma_0 y_2(a-\tau)} &\text{if } a\geq \tau,
\end{cases}
\end{equation*}
so that
$$\alpha \gamma_0 x_2(0)\int_\tau^\infty e^{-\mu_0 a-\gamma_0 y_2(a-\tau)}da=\delta,$$
due to \eqref{Eq:Sys_E2}, whence
$$x_2(0)=\dfrac{\delta}{\alpha \gamma_0}e^{\mu_0 \tau}(\mu_0+\gamma_0 y_2)=\beta_0 X^*.$$
It is then clear that
$$\lim_{t\to \infty}x(t,a)=x_2(a)$$
for every $a\geq 0$ and the result follows.
\end{proof}

\subsection{Stability of $E_2$}

In this section, we deal with the stability of $E_2$. Let the linear operator $\A:D(\A)\subset \Y\to \Y$, with $\Y=L^1(\R_+)\times \R$, be defined by
\begin{equation*}
D(\A)=\{(\phi,z)\in \Y, \phi \in W^{1,1}(\R_+) \mbox{ and } \phi(0)=\int_0^{\infty} \beta(a)\phi(a)da\},
\end{equation*}
\begin{equation*}
\A=\begin{pmatrix}
\mathcal{D} & 0 \\
0 & -\delta
\end{pmatrix}
\end{equation*}
where
\begin{equation*}
\mathcal{D} \phi=-\dfrac{d\phi}{da}-\mu \phi
\end{equation*}
and the function $h:\Y \to \Y$ given by
\begin{equation*}
h(\phi,z)=\begin{pmatrix}
- z \gamma(.)\phi(.) \\
\alpha z \int_0^{\infty} \gamma(a)\phi(a)\mathrm{d}a
\end{pmatrix}.
\end{equation*}
We know (see \cite{PerassoRichard19}, section 2.2) that $\A$ generates a positive $C_0$-semigroup. We denote by $D_{E_2}h$ the differential of $h$ at $E_2$ and we remind the following.
\begin{definition}
Let $\mathcal{L}(\Y)$ be the space of bounded linear operators on $\Y$ and let $\mathcal{K}(\Y)$ be the subspace of compact
operators on $\Y$. The essential norm $\Vert L\Vert _{\mathnormal{
ess}}$ of $L\in \mathcal{L}(\Y)$ is given by 
\begin{equation*}
\|L\|_{\mathnormal{ess}}=\underset{K\in \mathcal{K}(\Y)}{
\inf }\Vert L-K\Vert _{\Y}.
\end{equation*}
Let $\left\{ T_{\A}(t)\right\}_{t\geq 0} $ be a $C_{0}$-semigroup on $\Y$
with generator $\A:D(\A)\subset \Y\rightarrow \Y$ . The essential growth bound (or essential type) of $\left\{ T_{\A}(t)\right\}_{t\geq
0} $ is given by 
\begin{equation*}
\omega_{\mathnormal{ess}}(\A)=\underset{t\rightarrow \infty }{\lim }\dfrac{\ln (\|T_{\A}(t)\| _{\mathnormal{ess}})}{t}.
\end{equation*}
\end{definition}
We are ready to give the main result of this section.

\begin{theorem}\label{Thm:Stab_E2}
We have
$$\sigma(\A+D_{E_2}h)\subset \{z\in \C: \Re(z)\leq 0\},$$
where $\sigma(\A+D_{E_2}h)$ denotes the spectrum of $\A+D_{E_2}h$. Moreover 
$$\{z\in \sigma(\A+D_{E_2}h): \Re(z)=0\}\neq \emptyset$$
if and only if \eqref{Eq:Cond_Imag} holds, and in this case, the roots are given by \eqref{Eq:Lambda}. In particular, if \eqref{Eq:Cond_Imag} does not hold, then $E_2$ is locally asymptotically stable for \eqref{Eq:PDEmodel}.
\end{theorem}

\begin{proof}
We know from \cite[Theorem 3.3]{PerassoRichard19} that
$$\omega_\ess(\A+D_{E_2}h)<0.$$
Consequently, we have
$$\{\lambda \in \sigma(\A+D_{E_2}h) : \Re(\lambda)\geq 0\}\subset \sigma_p(\A+D_{E_2}h)$$
(see \cite{EngelNagel2000}, Corollary IV.2.11, p. 258), where $\sigma_p$ denotes the point spectrum. Similarly as in \cite[Section 3.2.3]{PerassoRichard19}, we look for solutions of the form $x(t,a)=\overline{x}(a)e^{\lambda t}$, $y(t)=\overline{y}e^{\lambda t}$, where the eigenvalue $\lambda \in \C$ has to satisfy the system $BY=C$, with:
\begin{equation*}
B=\begin{pmatrix}
b_1 & b_2 \\
b_3 & b_4
\end{pmatrix}, C=\begin{pmatrix}
0 \\
0
\end{pmatrix} \text{ and } Y=\begin{pmatrix}
\bar{x}(0) \\
\bar{y}
\end{pmatrix}, \text{ with}:
\end{equation*}
\begin{equation*}
\left\{
\begin{array}{rcl}
b_1 &=& 1-\int_0^{\infty}\beta(a)e^{-\int_0^a \left(\mu(s)+\lambda+ y^* \gamma(s)\right)\mathrm{d}s}\mathrm{d}a, \\
b_2 &=& \dfrac{\delta}{\alpha \Gamma} \int_0^{\infty}\beta(a) e^{-\int_0^a \left[\mu(s)+y^* \gamma(s)\right]\mathrm{d}s}\int_0^a \gamma(u) e^{-\lambda(a-u)}\mathrm{d}u\mathrm{d}a, \\
b_3 &=& \alpha y^* \int_0^{\infty} \gamma(a)e^{-\int_0^a (\mu(s)+\lambda+\gamma(s) y^*)\mathrm{d}s}\mathrm{d}a, \\
b_4 &=& -\lambda -\dfrac{\delta y^*}{\Gamma} \int_0^{\infty} \gamma(a)e^{-\int_0^a \left[\mu(s)+\gamma(s) y^*\right]\mathrm{d}s}\int_0^a \gamma(u) e^{-\lambda(a-u)}\mathrm{d}u\mathrm{d}a,
\end{array}
\right.
\end{equation*}
and $\Gamma=\int_0^{\infty} \gamma(a)e^{-\int_0^a\left[\mu(s)+y^* \gamma(s)\right]\mathrm{d}s}\mathrm{d}a$. While solving $BY=C$, one needs to have \\
$\det(B)=0$ to get a nonzero solution $Y$, that is equivalent to
\begin{equation}\label{Eq:Det-zero}
b_1b_4=b_2b_3.
\end{equation}
We see that
\begin{equation*}
\begin{array}{rcl}
\Gamma&=&\dfrac{\gamma_0 e^{-\mu_0 \tau}}{\mu_0 +y^*\gamma_0 }=\dfrac{\gamma_0}{\beta_0},
\end{array}
\end{equation*}
since
$$\mu_0+y^*\gamma_0 =\beta_0 e^{-\mu_0 \tau}.$$
Consequently, some computations lead to
$$b_1=1-\dfrac{\beta_0e^{-(\mu_0+\lambda)\tau}}{\mu_0+\lambda+y^*\gamma_0},$$
\begin{equation*}
\begin{array}{rcl}
b_2&=&\dfrac{\delta \beta_0 \gamma_0}{\alpha \lambda \Gamma} \displaystyle \int_\tau^\infty e^{\tau y^* \gamma_0}e^{-(\mu_0+y^*\gamma_0)a} \left(1-e^{\lambda(\tau-a)}\right)da \\
&=&\dfrac{\delta \beta_0 \gamma_0 e^{-\mu_0 \tau}}{\alpha \lambda \Gamma}\left(\dfrac{1}{\mu_0 +y^* \gamma_0}-\dfrac{1}{\mu_0+\lambda+y^* \gamma_0}\right) \\
&=&\dfrac{\delta \beta_0 \gamma_0 e^{-\mu_0 \tau}}{\alpha \Gamma (\mu_0+\lambda+y^*\gamma_0)(\mu_0+y^* \gamma_0)} \\
&=&\dfrac{\delta \beta_0}{\alpha (\mu_0+\lambda+y^* \gamma_0)},
\end{array}
\end{equation*}
$$b_3=\dfrac{\alpha y^* \gamma_0 e^{-(\mu_0+\lambda)\tau}}{\mu_0+\lambda+y^* \gamma_0},$$
and
\begin{equation*}
\begin{array}{rcl}
b_4&=&-\lambda-\dfrac{\delta y^* \gamma_0^2 e^{\gamma_0 y^* \tau}}{\lambda \Gamma}\displaystyle \int_\tau^\infty e^{-(\mu_0+\gamma_0 y^*)a}\left(1-e^{\lambda(\tau-a)}\right)da \\
&=&-\lambda-\dfrac{\delta y^* \gamma_0^2 e^{-\mu_0 \tau}}{\lambda \Gamma}\left(\dfrac{1}{\mu_0+y^* \gamma_0}-\dfrac{1}{\mu_0+\lambda+y^* \gamma_0}\right)\\
&=&-\lambda+\dfrac{\delta y^* \gamma_0^2 e^{-\mu_0 \tau}}{\Gamma(\mu_0+y^* \gamma_0)(\mu_0+\lambda+y^* \gamma_0)} \\
&=&-\lambda+\dfrac{\delta y^* \gamma_0}{\mu_0+\lambda+y^*\gamma_0 }.
\end{array}
\end{equation*}
Finally, \eqref{Eq:Det-zero} holds if and only if
\begin{equation*}
\begin{array}{rcrcl}
&&\left(1-\dfrac{\beta_0 e^{-(\mu_0+\lambda)\tau}}{\mu_0+\lambda+y^* \gamma_0}\right)\left(-\lambda-\dfrac{\delta y^* \gamma_0}{\mu_0+\lambda+y^*\gamma_0 }\right)&=&\dfrac{\delta \beta_0 y^*\gamma_0 e^{-(\mu_0+\lambda)\tau}}{(\mu_0+\lambda+y^* \gamma_0)^2} \\
&\Longleftrightarrow& -\lambda\left(1-\dfrac{\beta_0 e^{-(\mu_0+\lambda)\tau}}{\mu_0+\lambda+y^* \gamma_0}\right)&=&\dfrac{\delta  y^* \gamma_0}{\mu_0+\lambda+y^* \gamma_0} \\
&\Longleftrightarrow& -\lambda\left(\beta_0 e^{-\mu_0 \tau}+\lambda-\beta_0 e^{-(\mu_0+\lambda)\tau}\right)&=&\delta y^* \gamma_0,
\end{array}
\end{equation*}
i.e. if and only if \eqref{Eq:LocalStab} holds, where $\kappa_1$ and $\kappa_2$ are given by \eqref{Eq:kappa}. The result follows from Theorem \ref{Thm:LocStab}.	
\end{proof}
By means of Theorems \ref{Thm:PDE_Att} and \ref{Thm:Stab_E2}, we get the following result.

\begin{corollary}
Under the assumption \eqref{Eq:PartCase}, if 
$$\dfrac{\tau \sqrt{\delta y^* \gamma_0}}{2\pi}<1$$
then the equilibrium $E_2$ is globally asymptotically stable in $\Y_2$ for \eqref{Eq:PDEmodel}.
\end{corollary}

\begin{remark}
If 
$$\dfrac{\tau \sqrt{\delta y^* \gamma_0}}{2\pi}=1$$
then the attractiveness of $E_2$ in $\Y_2$ is ensured while the stability is not.
\end{remark}

\textbf{Acknowledgements.} The authors would like to thank Mostafa Adimy and Fabien Crauste about their relevant suggestions for this paper.

\bibliographystyle{siamplain}
\bibliography{references}

\end{document}